\documentclass{amsart}
\usepackage{amssymb}
\usepackage{graphicx}
\usepackage[usenames]{color}
\newtheorem{ass}{Assumption}

\def\cal{\mathcal}
\def\Pr{\mathbb{P}}
\def \R{I\!\!R}

\input{tcilatex}

\begin{document}
\title[decay/surge]{On decay-surge population models}
\author{Branda Goncalves, Thierry Huillet and Eva L\"ocherbach}
\address{B. Goncalves and T. Huillet: Laboratoire de Physique Th\'{e}orique et
Mod\'{e}lisation, CY Cergy Paris Universit\'{e}, CNRS UMR-8089, 2 avenue
Adolphe-Chauvin, 95302 Cergy-Pontoise, France \\
E-mails: branda.goncalves@outlook.fr, Thierry.Huillet@cyu.fr  \\
E. L\"ocherbach: SAMM, Statistique, Analyse et Mod\'elisation
Multidisciplinaire, Universit\'e Paris 1 Panth\'eon-Sorbonne, EA 4543 et FR
FP2M 2036 CNRS, France.\\
E-mail: eva.locherbach@univ-paris1.fr}
\maketitle
\begin{abstract}
We consider continuous space-time decay-surge population models which are 
semi-stochastic processes for which deterministically declining populations, 
bound to fade away, are reinvigorated at random times by bursts or 
surges of random sizes. In a particular separable framework (in a sense made precise below) we provide explicit formulae for the scale
(or harmonic) function and the speed measure of the process. The behavior of the scale function at infinity allows to formulate conditions under 
which such processes either explode or are transient at 
infinity, or Harris recurrent.  A description of the structures of both the 
discrete-time embedded chain and extreme record chain of such 
continuous-time processes is supplied.


\textbf{AMS Classification 2010:} 60J75, 92D25.
\end{abstract}

\section{Introduction}
This paper deals with decay-surge population models where a deterministically declining evolution following 
some nonlinear flow is interrupted by bursts of 
random sizes occurring at random times. Decay-surge models are natural models of many physical and biological phenomena, including the evolution of aging and declining populations which are reinvigorated by immigration, the height of the membrane potential of a neuron decreasing in between successive spikes of its presynaptic neurons due to leakage effects and jumping upwards after the action potential emission of its presynaptic partners, work processes in single-server queueing systems as from Cohen (1982)... Our preferred physical image will be the one of aging populations subject to immigration.

Decay-surge models have been extensively studied in the literature, see among others Eliazar and Klafter (2006), and Harrison and Resnick (1976) and (1978). Most studies concentrate however  on non-Markovian models such as shot noise or Hawkes processes, where superpositions of overlapping decaying populations are considered, see Br\'emaud and Massouli\'e (1996), Eliazar and Klafter (2007) and (2009), Huillet (2021), Kella and Stadje (2001), Kella (2009) and Resnick and Tweedie (1982).

Inspired by storage processes for dams, the papers of Resnick and Tweedie (1982), Kella (2009), \c{C}inlar and Pinsky (1971), Asmussen and Kella (1996), Boxma et al (2011),  Boxma at al (2006), Harrison and Resnick (1976), Brockwell et al (1982)  are mostly concerned with growth-collapse models when growth is from \emph{stochastic} additive inputs such as compound Poisson or L\'{e}vy processes or renewal processes. Here, the water level of a dam decreasing deterministically according to some fixed water release program is subject to sudden uprises due to rain or flood. Growth-collapse models are also very relevant in the Burridge-Knopoff stress-release model of earthquakes and continental drift, as from Carlson et al (1994), and in stick-slip models of interfacial friction, as from Richetti et al (2001).
As we shall see, growth-collapse models are in some sense 'dual' to decay-surge models.

In contrast with these last papers, and as in the works of Eliazar and Klafter (2006), (2007) and (2009), we concentrate in the present work on a \emph{deterministic} and continuous decay motion in between successive surges, described by a nonlinear  flow, determining the decay rate of the population and given by
$$ x_t (x) = x - \int_0^t \alpha ( x_s (x)) ds ,\; t \geq 0, \;  x_0 ( x) = x \geq 0.$$
In our process, upward jumps (surges) occur with state dependent rate $ \beta ( x), $ when the current state of the process is $x .$ When a jump occurs, the present size of the population $x$ is replaced by a new random value $ Y(x) > x, $ distributed according to some transition kernel $ K ( x, dy ) , y \geq x .$  

This leads to the study of a quite general family of 
continuous-time piecewise deterministic Markov processes (PDMP's) $X_t(x) $ representing the size of the population at time $t$ when 
started from the initial value $x \geq  0 .$ See Davis (1984). The infinitesimal generator of this process is given for smooth test functions by 
$$ {\mathcal G} u (x) = - \alpha (x) u'(x) + \beta ( x) \int_{ ( x, \infty) } K ( x, dy ) [ u(y) - u(x) ] , \; x \geq 0, $$
under suitable conditions on the parameters $ \alpha, \beta $ and $ K(x, dy ) $ of the process. In the sequel we focus on the study of separable kernels $K (x, dy ) $ where for each $ 0 \le x \le y , $  
\begin{equation}\label{eq:sep}
\int_{(y, \infty )} K( x,dz ) {=}\frac{k\left( y\right) }{k\left( x\right) }%
\text{,}
\end{equation}
for some positive non-increasing function  $k: [0, \infty ) \to [ 0, \infty ]  $ which is continuous on $ (0, \infty).$ 

The present paper proposes a precise characterization of the probabilistic properties of the above process in this separable frame. Supposing that $ \alpha (x) $ and $ \beta ( x) $ are continuous and positive on $ (0, \infty ) , $ the main ingredient of our study is the function 
\begin{equation}\label{eq:Gamma}
\Gamma (x) = \int_1^x \gamma ( y) dy, \mbox{ where } \gamma (y) = \beta (y) / \alpha (y) ,\;  y, x \geq 0 .
\end{equation}
Supposing that $\Gamma (\cdot ) $ is a {\bf space transform,} that is, $ \Gamma ( 0 ) = - \infty $ and $ \Gamma ( \infty ) =  \infty, $ we show that 
\begin{enumerate}
\item[1.]
Starting from some strictly positive initial value $ x > 0, $ the process does not get extinct (does not hit $0$) in finite time almost surely (Proposition \ref{prop:extinction}). In particular, imposing additionally $ k (0 ) < \infty ,$ we can study the process  in restriction to the state space $ (0, \infty). $ This is what we do in the sequel.
\item[2.]
The function 
\begin{equation}
 s(x) = \int_{1}^{x}\gamma (y)e^{-\Gamma (y)}/k(y)dy, \;  x \geq 0 , 
\end{equation} 
is a {\bf scale function} of the process, that is, solves $ {\mathcal G} s (x) = 0$ (Proposition \ref{prop:scale}). It is always strictly increasing and satisfies $ s (0) = -\infty  $ under our assumptions. But it might not be a space transform, that is, $ s( \infty ) $ can take finite values. 
\item[3.]
This scale function plays a key role in the understanding 
of the exit probabilities of the process and yields conditions under 
which the process either explodes in finite time or is transient at 
infinity. More precisely, if  $ s ( \infty ) < \infty ,$ we have the following explicit formula for the  exit probabilities. For any $ 0 < a < x < b , $
\begin{equation}\label{eq:exit1}
 {\mathbf{P}} ( X_t  \mbox{ enters  $(0, a] $ before entering $ [b, \infty ) $  } |X_0 = x )  =  \frac{s(x) - s(a) }{s(b) - s(a) }.
\end{equation} 
(Proposition \ref{prop:exit}).

Taking $ b \to \infty $ in the above formula, we deduce from this that $ s( \infty ) < \infty $ implies either that the process explodes in finite time (possesses an infinite number of jumps within some finite time interval) or  that it is transient at infinity. 

Due to the asymmetric dynamic of the process (continuous motion downwards and jumps upwards such that entering the interval $ [b, \infty ) $ starting from $x < b $ always happens by a jump), \eqref{eq:exit1} does not hold if $ s (\infty ) = \infty .$
\item[4.]
Imposing additionally that $ \beta ( 0 ) > 0 ,$ {\bf Harris recurrence (positive or null)} of the process is equivalent to the fact that $s$ is a space transform, that is, $ s( \infty ) = \infty $ (Theorem \ref{theo:harris}). 

In this case, up to constant multiples, the unique invariant measure of the process possesses a Lebesgue density (speed density) given by 
$$ \pi ( x) = \frac{ k (x) e^{ \Gamma ( x) } }{ \alpha ( x) }, x > 0 .$$ 
More precisely, we show how the scale function can be used to obtain 
Foster-Lyapunov criteria in the spirit of Meyn and Tweedie (1993) implying the non-explosion of the process together with its recurrence under 
additional irreducibility properties. Additional conditions, making use of the speed measure, under 
which first moments of hitting times are finite, are also supplied in 
this setup.
\end{enumerate} 

{\bf Organization of the paper.} In Section 2, we introduce our model and state some first results. Most importantly, we establish a simple relationship between decay-surge models and growth-collapse models as studied in Goncalves, 
Huillet and L\"ocherbach (2021) that allows us to obtain explicit representations of the law of the first jump time and of the associated speed measure without any further study. Section 3 is devoted to the proof of the existence of the scale function (Proposition \ref{prop:scale}) together with the study of first moments of hitting times which are shown to be finite if the speed density is integrable at $ + \infty $ (Proposition \ref{prop:hitting}). Section 4 then collects our main results. If the scale function is a space transform, it can be naturally transformed into a Lyapunov function in the sense of Meyn and Tweedie (1993) such that the process does not explode in finite time and comes back to certain compact sets infinitely often (Proposition \ref{prop:MT}). Using the regularity produced by the jump heights according to the absolutely continuous transition kernel $ K( x, dy ), $ Theorem \ref{theo:petite} establishes then a local Doeblin lower bound on the transition operator of the process - a key ingredient to prove Harris recurrence which is our main result Theorem \ref{theo:harris}.  Several examples are supplied including one related to linear Hawkes processes and to shot-noise processes. In a last part of the work, we focus on the embedded chain of  the process, sampled at the jump times, 
which, in addition to its fundamental relevance, is easily amenable to simulations. Following Adke (1993), we also draw the attention to the structure of 
the extreme record chain of  $X_t(x) $, allowing in particular to derive 
the distribution of the first upper record time and overshoot value, as 
a level crossing time and value. This study is motivated by the understanding of the time of the first crossing of some high population level and the amount of the corresponding overshoot, as, besides extinction, populations can face overcrowding.

\section{The model, some first results and a useful duality property}
We study population decay models with random surges described by a Piecewise Deterministic Markov Process (PDMP) $ X_t  , t \geq 0,$ starting from some initial value $ x \geq 0$ at time $0$ and taking values in $ [0, \infty ) .$ The main ingredients of our model are 
\begin{enumerate}
\item[1.]
The drift function $\alpha (x) .$ We suppose that $ \alpha : 
[ 0,\infty ) \to [0, \infty )  $ is continuous, with $ \alpha ( x) > 0 $ for all $ x > 0. $ In between successive jumps, the process follows the decaying dynamic 
\begin{equation}\label{eq:detdyn}
\overset{.}{x}_{t} (x) =-\alpha \left(
x_{t} (x) \right) , x_{0} (x) =x \geq 0.
\end{equation}
\item[2.]
The jump rate function $ \beta ( x) .$ We suppose that $ \beta : {( 0 , \infty)} \to [0, \infty ) $ is continuous and
$ \beta ( x) > 0 $ for all $ x > 0.$ 
\item[3.]
The jump kernel $K (x, dy ).  $ This is a transition kernel from $ [0, \infty) $ to $[0, \infty ) $ such that for any $ x >  0, $ $ K ( x, [x, \infty) ) = 1 .$ Writing $ K(x, y ) = \int_{(y, \infty) } K (x, dz), $ we suppose that $K(x, y ) $ is jointly continuous
in $x$ and $y.$

\end{enumerate}

In between successive jumps, the population size follows the deterministic flow $ x_t (x) $ given in \eqref{eq:detdyn}.  
For any $0\leq a<x$, the integral 
\begin{equation}\label{eq:timetoa}
t_{a}\left( x\right) :=\int_{a}^{x}\frac{dy}{\alpha \left( y\right) }
\end{equation}
is the time for the flow to hit $a$ starting from $x.$ In particular, starting from $ x > 0, $ the flow reaches $0$ after
some time $t_0\left( x\right) =\int_{0}^{x}\frac{dy}{\alpha \left(
y\right) }\leq \infty $. We refer to \cite{DSGHL} for a variety of examples of such decaying flows that can hit zero in finite time or not.

Jumps occur at state dependent
rate $\beta \left( x\right) .$ At the jump times, the size of the population
grows by a random amount $\Delta \left( X_{t-}\right) >0$ of its current
size $X_{t-}.$ Writing $Y\left( X_{t-}\right) :=X_{t-}+\Delta
\left( X_{t-}\right) $ for the position of the process right after its jump, $ Y ( X_{t-}) $ is distributed according to $ K ( X_{t- } , d y ).$  

Up to the next jump time, $X_{t}$ then decays again, following the
deterministic dynamics \eqref{eq:detdyn}, started at the new value $Y\left( X_{t-}\right) :=X_{t-}+\Delta
\left( X_{t-}\right) $.

We are thus led to consider the PDMP $ X_t$ 
 with state-space $\left[ 0,\infty \right) $ solving 
\begin{equation}
dX_{t}=-\alpha \left( X_{t}\right) dt+\Delta \left( X_{t-}\right)
\int_{0}^{\infty }\mathbf{1}_{\left\{ r\leq \beta \left( X_{t-}\left(
x\right) \right) \right\} }M\left( dt,dr\right) ,\text{ }X_{0}=x,  \label{C0}
\end{equation}
where $M\left( dt,dr\right) $ is a Poisson measure on $\left[ 0,\infty
\right) \times \left[ 0,\infty \right) $. Taking $dt\ll 1$ be the system's
scale, this dynamics means alternatively that we have transitions
\begin{eqnarray*}
X_{t-} &=&x\rightarrow x-\alpha \left( x\right) dt\text{ with probability }%
1-\beta \left( x\right) dt \\
X_{t-} &=&x\rightarrow x+\Delta \left( x\right) \text{ with probability }%
\beta \left( x\right) dt.
\end{eqnarray*}

It is a nonlinear version of the Langevin equation with jumps.

\subsection{Discussion of the jump kernel}
We have 
\begin{equation*}
\mathbf{P}\left( Y\left( x\right) >y\mid X_{t-}=x\right) =K\left( x,y\right) = \int_{ ( y, \infty ) } K (x, dz).
\end{equation*}
Clearly $K\left( x,y\right) $\ is a non-increasing function of $y$\ for all $%
y\geq x$ satisfying $K\left( x,y\right) =1$\ for all $y<x$\textbf{. }By continuity, this implies that $%
K\left( x,x\right) =1$ such that the law of $\Delta \left( x\right) $ has
no atom at $0.$

In the sequel we concentrate on the separable case 
\begin{equation}\label{eq:sep}
K\left( x,y\right) {=}\frac{k\left( y\right) }{k\left( x\right) }%
\text{,}
\end{equation}
where $k: [0, \infty ) \to [ 0, \infty ]  $ is any positive non-increasing function. In what follows we suppose that $k$ is continuous and finite on $ (0, \infty). $ 

Fix $z>0$ and assume $y=x+z.$ Then  
$$
{\bf P}\left( Y\left( x\right) >y\right) =\frac{k\left( x+z\right) }{k\left(
x\right) } .
$$
Depending on $k\left( x\right) , $ this probability can be a
decreasing or an increasing function of $x$ for each $z.$

\begin{example}
Suppose $k\left( x\right) =e^{-x^{\alpha }}${\bf , }$\alpha
>0,${\bf \ }$x\geq 0$ (a Weibull distribution).

- If $\alpha <1,$ then $\partial _{x}K (x, x + z ) >0,$ so that the larger $%
x,$ the larger  ${\bf P}\left( Y\left( x\right) >x+z\right) .$ In other words, if 
the population stays high, the probability of a large number of immigrants
will be enhanced. There is a positive feedback of $x$ on $\Delta
\left( x\right) , $ translating a herd effect.

- If $\alpha =1,$ then $\partial _{x}K(x, x+z) =0$ and there is no
feedback of $x$ on the number of immigrants, which is then exponentially
distributed.

- If $\alpha >1,$ then $\partial _{x}K (x, x+z) <0$ and the larger  $x,$ the smaller the probability ${\bf P}\left( Y\left( x\right)
>x+z\right) .$ In other words, if the population stays high, the probability of a large
number of immigrants will be reduced. There is a negative feedback of  $x$
on $\Delta \left( x\right) .$ 
\end{example}

\begin{example}
{\bf The case $k\left( 0\right) <\infty . $} Without loss of generality, we may take $k\left( 0\right) =1.$
Assume that $k\left( x\right) ={\bf P}\left( Z>x\right) $ for some proper random variable $Z>0$ and that 
\[
Y\left( x\right) \stackrel{d}{=}Z\mid Z>x 
\]
so that $Y\left( x\right) $ is amenable to the truncation of $Z$ above $x.$
Thus
\[
{\bf P}\left( Y\left( X_{t-}\right) >y\mid X_{t-}=x\right) =\frac{{\bf P}%
\left( Z>y,Z>x\right) }{{\bf P}\left( Z>x\right) }=\frac{k\left( y\right) }{%
k\left( x\right) },\text{ for }y>x .
\]
A particular (exponential) choice is 
\[
k\left( x\right) =e^{-\theta x},\text{ }\theta >0, 
\]
with ${\bf P}\left( Y\left( X_{t-}\right) >y\mid X_{t-}=x\right) =e^{-\theta
\left( y-x\right) }$ depending only on $y-x$. Another possible one is
(Pareto): $k\left( x\right) =\left( 1+x\right) ^{-c},$ $c>0$.

Note $K\left( 0,y\right) =k\left( y\right) >0$ for all $y>0,$ and $k\left(
y\right) $ turns out to be the cpdf of a jump above $y$, starting
from $0$: state $0$ is {\bf reflecting}.

{\bf The case $ k (0 ) = \infty .$}
Consider $k\left( x\right)
=\int_{x}^{\infty }\mu \left( Z\in dy\right) $ for some positive Radon
measure $\mu $ with infinite total mass. In this case, 
\[
{\bf P}\left( Y\left( X_{t-}\right) >y\mid X_{t-}=x\right) =\frac{\mu \left(
Z>y,Z>x\right) }{\mu \left( Z>x\right) }=\frac{k\left( y\right) }{k\left(
x\right) },\text{ for }y>x. 
\]

Now, $K\left( 0,y\right) =0$ for all $y>0$ and state $0$ becomes 
{\bf attracting}. An example is $k\left( x\right) =x^{-c},$ $c>0,$ which is
not a cpdf.

The ratio $k\left( y\right) /k\left( x\right) $ is thus the conditional
probability that a jump is greater than the level $y$ given that it did
occur and that it is greater than the level $x$, see Eq. (1) of Eliazar and Klafter (2009) for a similar choice.

Our motivation for choosing the separable form $K\left(
x,y\right) =k\left( y\right) /k\left( x\right) $ is that it accounts for the
possibility of having state $0$ either absorbing or reflecting for upwards
jumps launched from $0$ and also that it can account for both a
negative or positive feedback of the current population size on the number
of incoming immigrants.
\end{example}

\begin{example}
One can think of many other important and natural choices of $K\left( x,y\right) ,$ not in
the separable class, among which those for which 
\[
K\left( x,dy\right) =\delta _{Vx}\left( dy\right) 
\]
for some random variable $V>1.$ For this class of kernels, state $0$ is {\bf always}
attracting. For example, choosing $V=1+E$ where:

1/ $E$ is exponentially distributed with pdf $e^{-\theta x},$ $Y\left(
x\right) =Vx$ yields 
\[
{\bf P}\left( Y\left( x\right) >y\mid X_{-}=x\right) =K\left( x,y\right) =%
{\bf P}\left( \left( 1+E\right) x>y\right) =e^{-\theta \left( \frac{y}{x}%
-1\right) }, 
\]

2/ $E$ is Pareto distributed with pdf $\left( 1+x\right) ^{-c},$ $c>0,$ $%
Y\left( x\right) =Vx$ yields 
\[
{\bf P}\left( Y\left( x\right) >y\mid X_{-}=x\right) =K\left( x,y\right) =%
{\bf P}\left( \left( 1+E\right) x>y\right) =\left( y/x\right) ^{-c},
\]
both with $K\left( 0,y\right) =0.$ When \\
3/: $V\sim \delta _{v},$ $v>1$, then 
$K\left( x,y\right) ={\bf 1}\left( y/x\leq v\right) .$ The three kernels
depend on $y/x.$

Note that in all three cases, $\partial _{x}K (x, x+z) >0, $ so that the larger %
$x,$ the larger ${\bf P}\left( Y\left( x\right) >x+z\right). $ If the population stays high, the probability of a large number of
immigrants will be enhanced. There is a positive feedback of $x$  on $%
\Delta \left( x\right) ,$ translating a herd effect.
\end{example}

\begin{remark}
A consequence of the separability condition of $K$ is the following. Consider a
Markov sequence of after-jump positions defined recursively by $%
Z_{n}=Y\left( Z_{n-1}\right) $, $Z_{0}=x_{0}$. With $x_{m}>x_{m-1}$, we have 
\[
{\bf P}\left( Z_{m}>x_{m}\mid Z_{m-1}=x_{m-1}\right) =K\left(
x_{m-1},x_{m}\right) \text{, }m=1,...,n,
\]
so that, with $x_{0}<x_{1}<...<x_{n}$, and under the separability condition
on $K$, the product
\[
\prod_{m=1}^{n}{\bf P}\left( Z_{m}>x_{m}\mid Z_{m-1}=x_{m-1}\right)
=\prod_{m=1}^{n}K\left( x_{m-1},x_{m}\right) =\prod_{m=1}^{n}\frac{k\left(
x_{m}\right) }{k\left( x_{m-1}\right) }=\frac{k\left( x_{n}\right) }{k\left(
x_{0}\right) }
\]
only depends on the initial and terminal states $\left( x_{0},x_{n}\right) $
and not on the full path $\left( x_{0},...,x_{n}\right) .$ 
%Note (by summing
%over the permutations of $\left( x_{1},...,x_{n-1}\right) $) that the $n$%
%-step transition probability 
%\begin{eqnarray*}
%{\bf P}\left( Z_{n}>x_{n}\mid Z_{0}=x_{0}\right)  &=&\int_{\left\{
%x_{0}<x_{1}<...<x_{n-1}<x_{n}\right\} }\prod_{m=1}^{n}{\bf P}\left(
%Z_{m}>x_{m}\mid Z_{m-1}=x_{m-1}\right) dx_{1}...dx_{n-1} \\
%&=&\frac{k\left( x_{n}\right) }{k\left( x_{0}\right) }\frac{\left(
%x_{n}-x_{0}\right) ^{n-1}}{\left( n-1\right) !}.
%\end{eqnarray*}
%also only depends on $\left( x_{0},x_{n}\right) $ but is not separable$.$%
%\newline
\end{remark}

\subsection{The infinitesimal generator} 
In what follows we always work with separable kernels. Moreover, we write $X_t ( x) $ for the process given in \eqref{C0} to emphasize the dependence on the starting point $ x , $ that is, $X_t (x) $ designs the process with the above dynamics \eqref{C0} and satisfying $ X_0 ( x) = x .$ 
If the value of the starting point $x$ is not important, we shall also write $ X_t $ instead of $ X_t ( x) .$ 

Under the separability condition, the infinitesimal generator of $X_{t}$ acting on bounded smooth test functions $u$ takes the following simple form 
\begin{equation}\label{eq:generator}
\left( \mathcal{G}u\right) \left( x\right) =-\alpha \left( x\right)
u^{\prime }\left( x\right) +\frac{\beta \left( x\right) }{k\left( x\right) }%
\int_{x}^{\infty }k\left( y\right) u^{\prime }\left( y\right) dy , x \geq 0. 
\end{equation}

\begin{remark}
In Eliazar and Klafter (2006), a particular scale-free version of
decay-surge models with $\alpha \left( x\right) \propto x^{a}$, $\beta
\left( x\right) \propto x^{b}$ and $k\left( x\right) \propto x^{-c}$, $c>0$,
has been investigated.
\end{remark}

\begin{remark}
If $x_{t}$ goes extinct in finite time $t_{0}(x) < \infty ,$ since $x_{t}$ is
supposed to represent the size of some population, we need to impose $x_{t}=0
$ for $t\geq t_0(x),$ forcing state $0$ to be absorbing. From this time
on, $X_{t}$ can re-enter the positive orthant if there is a positive
probability to move from $0$ to a positive state meaning $ k (0 ) < \infty $ and $\beta \left( 0\right) >0.$ 
In such a case, the first
time $X_{t}$ hits state $0$ is only a first local extinction time the
expected value of which needs to be estimated. The question of the time
elapsed between consecutive local extinction times (excursions) also arises.

On the contrary, for situations for which $k(0)=\mathbf{\infty }$ or $\beta
\left( 0\right) =0,$ the first time $X_{t}$ hits state $0$ will be a global
extinction time.
\end{remark}

\subsection{Relation between decay-surge and growth-collapse processes}
In this subsection, we exhibit a natural relationship between decay-surge population models, as studied here, and growth-collapse models as developed in Boxma et al (2006), Goncalves et al (2021), Gripenberg (1983) and Hanson and Tuckwell (1978). Growth-collapse models describe deterministic population growth where at random jump times the size of the population undergoes a catastrophe reducing its current size to a random fraction of it. More precisely, the generator of a growth-collapse process, having parameters $ ( \tilde \alpha, \tilde \beta, \tilde h ) ,$ is given for all smooth test functions by 
\begin{equation}\label{eq:Gsep}
(\tilde {\mathcal G} u)  (x) = \tilde \alpha (x) u' ( x) - \tilde \beta ( x)/\tilde h(x)  \int_0^x u'(y)\tilde h(y) dy   , x \geq 0.
\end{equation} 
In the above formula, $ \tilde \alpha, \tilde \beta $ are continuous and positive functions on $ (0, \infty), $ and $ \tilde h $ is positive and non-decreasing on $ (0, \infty).$  

In what follows, consider a decay-surge process $X_{t}$ defined by the
triple $\left( \alpha ,\beta ,k\right) $ and let $\widetilde{X}_{t}=1/X_{t}.$

\begin{proposition}
The process $\widetilde{X}_{t}$  is a growth-collapse
process as studied in \cite{GHL} with triple $\left( \widetilde{\alpha },%
\widetilde{\beta },\widetilde{h}\right) $ given by 
$$
\widetilde{\alpha }\left( x\right)  =x^{2}\alpha (1/x) ,
\widetilde{\beta }(x) =\beta (1/x)  \mbox{ and } 
\widetilde{h}(x) =k(1/x), \; x > 0.
$$
\end{proposition}

\begin{proof}
Let $u$ be any smooth test function and study $u ( \tilde X_t) = u \circ g ( X_t) $ with $g ( x) = 1/x . $ By Ito's formula for processes with jumps,  
$$ u ( \tilde X_t ) = u \circ g (X_t ) = u ( \tilde X_0 ) + \int_0^t {\mathcal G} (u \circ g ) (X_{t'}) dt' + M_t, $$ where $M_t$ is a local martingale. We obtain 
\begin{eqnarray*}
{\mathcal G} (u \circ g ) (x)  &=& -\alpha(x)(u \circ g)'(x)+ \frac{\beta(x)}{k(x)}\int_x^{\infty} (u \circ g)'(y) k(y) dy, \\
												&=& \frac{1}{x^2}\alpha(x)u'(\frac{1}{x}) + \frac{\beta(x)}{k(x)}\int_x^{\infty} u'(\frac{1}{y}) k(y) \frac{-dy}{y^2} 
\end{eqnarray*}
Using the change of variable $y=1/x ,$ this last expression can be rewritten as  
$$
\frac{1}{x^2}\alpha(x)u'(\frac{1}{x}) -\frac{\beta(x)}{k(x)}\int_0^{1/x} u'(z) k(\frac{1}{z}) dz
= \widetilde{\alpha }(y)u^{\prime }(y)-\frac{\widetilde{\beta }(y)}{\widetilde{h}(y)}\int_{0}^{y}u^{\prime }(t)\widetilde{h}(t)dt 
$$
which is the generator of the process $\widetilde{X}_{t}.$ 
\end{proof}
In what follows we speak about the above relation between the decay-surge (DS) process $ X$ and the growth-collapse (GC) process $ \tilde X$ as {\it DS-GC-duality}. Some simple properties of the process $X$ follow directly from the above duality as we show next. Of course, the above duality does only hold up to the first time one of the two processes leaves the interior $ (0, \infty ) $ of its state space. Therefore a particular attention has to be paid to state $0$ for $X_t$ or equivalently to state $ + \infty $ for $\tilde X_t.$ Most of our results will only hold true under conditions ensuring that, starting from $ x > 0, $ the process $X_t$ will not hit $0$ in finite time.
 
Another important difference between the two processes is that the simple transformation $ x \mapsto 1/x$ maps a priori unbounded sample paths $ X_t $ into bounded onces $ \tilde X_t $ (starting from $ \tilde X_0 = 1/x, $ almost surely, $ \tilde X_t \le \tilde x_t ( 1/x) -$ a relation which does not hold for $X$).

\subsection{First consequences of the DS-GC-duality}
Given $X_{0}=x > 0$, the first jump times both of the DS-process $X_t$ starting from $x,$ and of the GC-process $ \tilde X_t, $ starting from $ 1/x,$ coincide and are given by 
\begin{equation*}
T_{x}=\inf \{ t >  0 : X_t \neq X_{t-}  | X_0 = x  \}= \tilde T_{\frac1x}= \inf \{ t>0:\tilde X_{t}\neq \tilde X_{t-}|  \tilde X_{0}=\frac1x\}.
\end{equation*}
Introducing 
\begin{equation*}
\Gamma \left( x\right) :=\int_1^{x}\gamma \left( y\right) dy, \mbox{ where } \gamma \left( x\right) :=\beta \left( x\right) /\alpha \left( x\right) , x > 0 , 
\end{equation*}
and the corresponding quantity associated to the process $ \tilde X_t,$  
$$ \widetilde{\Gamma } \left( x\right) = \int_1^x \tilde \gamma \left( y\right)dy , \; \tilde \gamma \left( x\right)= \tilde \beta \left( x\right)/ \tilde \alpha \left( x\right) , x > 0 ,$$
clearly, $\widetilde{\Gamma }\left( x\right) =-\Gamma \left( 1/x\right) $ for all $ x > 0.$ 

Arguing as in Sections 2.4 and 2.5 of  \cite{GHL}, a direct consequence of the above duality is the fact that for all $t < t_0 ( x), $  
\begin{equation}  \label{eq:tx}
\mathbf{P}\left( T_{x} >t\right) =e^{-\int_{0}^{t}\beta
\left( x_{s}\left( x\right) \right) ds}=e^{-\left[ \Gamma \left( x\right)
-\Gamma \left( x_{t}\left( x\right) \right) \right] }.
\end{equation}

To ensure that $\mathbf{P}\left( T_{x} <\infty \right) =1,$ in accordance with Assumption 1 of \cite{GHL} we will impose the condition 
\begin{ass}\label{As1}
$\Gamma \left(0
\right) =- \infty .$
\end{ass}

\begin{proposition}\label{prop:extinction}
Under Assumption \ref{As1},  the stochastic
process $X_{t}(x) , x > 0 , $ necessarily jumps before reaching $0.$ In particular, for any $ x > 0, $ $ X_t(x)$ almost surely never reaches $0$ in finite time. 
\end{proposition}
\begin{proof}
By duality, we have 
$$ {\mathbf{P}} (X \mbox{ jumps before reaching $0$} | X_0 = x ) = {\mathbf{P}} (  \tilde X \mbox{ jumps before reaching $ + \infty $}| \tilde X_0 = \frac1x) = 1, $$
as has been shown in Section 2.5 of \cite{GHL}, and this implies the assertion. 
\end{proof}

In particular, the only situation where the question of the extinction of the process $X $ makes sense (either local or total) is when $t_0(x)<\infty $ and $\Gamma(0)>-\infty. $ 

\begin{example}
We give an example where finite time extinction of the process is possible. 
Suppose $\alpha \left( x\right) =\alpha_1 x^{a}$ with $\alpha_1>0$ and $a<1$. Then $x_{t}\left(
x\right) ,$ started at $x$, hits $0$ in finite time $t_{0}\left( x\right)
=x^{1-a}/\left[ \alpha_1 \left( 1-a\right) \right] $, with 
\[
x_{t}\left( x\right) =\left( x^{1-a}+\alpha_1 \left( a-1\right) t\right)
^{1/\left( 1-a\right) },
\]
see \cite{DSGHL}. Suppose $\beta \left( x\right) =\beta_1 >0$, constant. Then, with $\gamma_1=\beta_1/\alpha_1>0$  
\[
\Gamma \left( x\right) =\int_{1}^{x}\gamma \left( y\right) dy=\frac{\gamma_1 }{%
1-a}\left( x^{1-a}-1\right) 
\]
with $\Gamma \left( 0\right) =-\frac{\gamma_1 }{1-a}>-\infty .$ 
Assumption \ref{As1} is not fulfilled, so $X$ can hit $0$ in finite time and there
is a positive probability that $T_x  =\infty $. On this last
event, the flow $x_{t}\left( x\right) $ has all the time necessary to first
hit $0$ and, if in addition the kernel $k$ is chosen so as $k(0)=\infty ,$
to go extinct definitively. The time of extinction $\tau \left( x\right) $
of $X$ itself can be deduced from the renewal equation in distribution 
\[
\tau \left( x\right) \stackrel{d}{=}t_{0}\left( x\right) {\bf 1}_{\{T_x =\infty \}}+\tau ^{\prime }\left( Y\left( x_{T_x   }\left(
x\right) \right) \right) {\bf 1}_{\{T_x <\infty \}}
\]
where $\tau ^{\prime }$ is a copy of $\tau $.

We conclude that for this family of models, $X$ itself goes extinct in
finite time. This is an interesting regime that we shall not investigate any
further.
\end{example}

Let us come back to the discussion of Assumption \ref{As1}. 
It follows immediately from Eq. (13) in \cite{GHL} that for $x>0,$ under Assumption \ref{As1} and supposing that $ t_0 ( x) = \infty $ for all $ x > 0 ,$ 
\begin{equation*}
\mathbf{E}\left( T_{x} \right) =e^{-\Gamma \left( x\right) }\int_{0}^{x}\frac{dz}{\alpha \left( z\right) 
}e^{\Gamma \left( z\right) }.
\end{equation*}
Clearly, when $x\rightarrow 0$, $\mathbf{E}\left( T_{x} 
\right) \sim 1/\alpha \left( x\right) $.

\begin{remark}
$\left( i\right) $ If $\beta \left( 0\right) >0$ then Assumption \ref{As1} implies 
$t_{0}(x)=\infty ,$ such that $0$ is not accessible.

$\left( ii\right) $ Notice also that $t_0 (x) < \infty $ together with  Assumption \ref{As1} implies that $ \beta ( 0 ) = \infty $ such that 
the process $ X_t( x) $ is prevented from hitting $0$ even though $x_t(x) $ reaches it in finite time due to the fact that the jump rate $\beta (x) $ blows up as $x \to 0.$
\end{remark}

\subsection{Classification of state $0$}

Recall that for all $ x > 0, $ 
\begin{equation*}
t_{0}(x{)}=\int_{0}^{x}\frac{dy}{\alpha \left( y\right) }
\end{equation*}
represents the time required for $x_{t}$ to move from $x>0$ to $0$. So: 
\begin{eqnarray*}
&\text{ If } t_{0}(x{)}<\infty \mbox{ and } \Gamma ( 0 ) > - \infty , &\text{ state }0\text{ is accessible.} \\
&\text{ If } t_{0}(x{)}=\infty  \mbox{ or } \Gamma ( 0 ) = - \infty , &\text{ state }0\text{ is inaccessible.}
\end{eqnarray*}
We therefore introduce the following conditions which apply in the separable
case $K(x,y)=k(y)/k(x).$

\textbf{Condition (R):} $(i)$ $\beta \left( 0\right) >0$ and \newline
$\left( ii\right) $ $K\left( 0,y\right) =k\left( y\right) /k\left( 0\right)
>0 $ for some $y>0$ (and in particular $k\left( 0\right) <\infty) .$

\textbf{Condition (A):} $\frac{\beta(0)}{k(0)}k(y)=0 $ for all 
$y > 0 $. 

State $0$ is reflecting if condition $\left( R\right) $ is satisfied and it
is absorbing if condition (A) is satisfied.

This leads to four possible combinations for the boundary state $0$:

\textbf{Condition} $\left( R\right) $ and $t_{0}(x)<$\textbf{\ }$\infty $ and $ \Gamma (0 ) > - \infty :$ 
regular (reflecting and accessible).

\textbf{Condition }$\left( R\right) $\ and $t_{0}(x)=$\textbf{\ }$\infty $ or $ \Gamma ( 0 ) = - \infty :$
entrance (reflecting and inaccessible).

\textbf{Condition }$\left( A\right) $ and $t_{0}(x)<$\textbf{\ }$\infty $ and $ \Gamma (0 ) > - \infty :$ 
exit (absorbing and accessible).

\textbf{Condition }$\left( A\right) $ and $t_{0}(x)=$\textbf{\ }$\infty $ or $ \Gamma ( 0 ) = - \infty :$ natural
(absorbing and inaccessible).

\subsection{Speed measure.}
Suppose now an invariant measure (or speed measure) $\pi \left( dy\right) $
exists. Since we supposed $\alpha (x)>0$ for all $x>0,$ we necessarily have $%
x_{\infty }(x)=0$ for all $x>0$ and so the support of $\pi $ is $[
x_{\infty }(x)=0,\infty) .$ Thanks to our duality relation, by Eq. (19) of \cite{GHL} 
the explicit expression of the speed measure is given by $\pi ( dy) = \pi ( y) dy $ with 
\begin{equation}
\pi \left( y\right) =C\frac{k\left( y\right) e^{\Gamma \left( y\right) }}{%
\alpha \left( y\right) },  \label{C5}
\end{equation}
up to a multiplicative constant $C>0$. If and only if this function is
integrable at 0 and $\infty ,$ $\pi \left( y\right) $ can be tuned to a probability
density function. 

\begin{remark}
$\left( i\right) $When $k\left( x\right) =e^{-\kappa _{1}x}$, $\kappa _{1}>0$%
, $\alpha \left( x\right) =\alpha _{1}x$ and $\beta \left( x\right) =\beta
_{1}>0$ constant, $\Gamma \left( y\right) =\gamma _{1}\log y$, $\gamma
_{1}=\beta _{1}/\alpha _{1}$ and 
\begin{equation*}
\pi \left( y\right) =Cy^{\gamma _{1}-1}e^{-\kappa _{1}y}
\end{equation*}
a Gamma$\left( \gamma _{1},\kappa _{1}\right) $ distribution. This result is
well-known, corresponding to the linear decay-surge model (a jump version of
the damped Langevin equation) having an invariant (integrable) probability
density, see Malrieu (2015). We shall show later that the corresponding process $X $ is
positive recurrent.

$\left( ii\right) $A less obvious power-law example is as follows: Assume $%
\alpha \left( x\right) =\alpha _{1}x^{a}$ ($a>1$) and $\beta \left( x\right)
=\beta _{1}x^{b}$, $\alpha _{1},\beta _{1}>0$ so that $\Gamma \left(
y\right) =\frac{\gamma _{1}}{b-a+1}y^{b-a+1}.$ We have $\Gamma \left(
0\right) =-\infty $ if we assume $b-a+1=-\theta $ with $\theta >0$, hence $\Gamma
\left( y\right) =-\frac{\gamma _{1}}{\theta }y^{-\theta }.$ Taking $k\left(
y\right) =e^{-\kappa _{1}y^{\eta }}$, $\kappa _{1},\eta >0$%
\begin{equation*}
\pi \left( y\right) =Cy^{-a}e^{-\left( \kappa _{1}y^{\eta }+\frac{\gamma _{1}%
}{\theta }y^{-\theta }\right) }
\end{equation*}
which is integrable both at $y=0$ and $y=\infty $. %The corresponding process 
%$\left\{ X\right\} $ is also positive recurrent. 
As a special case, if $a=2$
and $b=0$ (constant jump rate $\beta \left( x\right) $), $\eta =1$ 
\begin{equation*}
\pi \left( y\right) =Cy^{-2}e^{-\left( \kappa _{1}y+\gamma _{1}y^{-1}\right)
},
\end{equation*}
an inverse Gaussian density.

$\left( iii\right) $ In Eliazar and Klafter (2007) and (2009), a special case of our
model was introduced for which $k\left( y\right) =\beta \left( y\right) $.
In such cases, 
\begin{equation*}
\pi \left( y\right) =C\gamma \left( y\right) e^{\Gamma \left( y\right) }
\end{equation*}
so that 
\begin{equation*}
\int_{0}^{x}\pi \left( y\right) dy=C\left( e^{\Gamma \left( x\right)
}-e^{\Gamma \left( 0\right) }\right) =Ce^{\Gamma \left( x\right) },
\end{equation*}
under the Assumption $\Gamma \left( 0\right) =-\infty $. If in addition $%
\Gamma \left( \infty \right) <\infty $, $\pi \left( y\right) $ can be tuned
to a probability density.\newline
\end{remark}

\section{Scale function and hitting times}

In this section we start by studying the scale function of $X_t,$ before switching to hitting times features that make use of it.
A scale function $s\left( x\right) $ of the process is any function solving $%
\left( \mathcal{G}s\right) \left( x\right) =0 .$  In other words, a scale function is a function that transforms the process into a local martingale. Of course, any constant function is solution. Notice that for the growth-collapse model considered in \cite{GHL}, other scale functions than the constant ones do not exist.

In what follows we are interested in non-constant solutions and conditions ensuring the existence of those. 
To clarify ideas,  we introduce the following condition
\begin{ass}\label{ass:C}
Let 
\begin{equation}\label{eq:s2}
s (x)=\int_{1}^{x}\gamma (y)e^{-\Gamma (y)}/k(y)dy , \; x \geq 0, 
\end{equation}
and suppose that $ s ( \infty ) = \infty .$ 
\end{ass}
Notice that Assumption \ref{ass:C} implies that $ k( \infty ) = 0 $ which is reasonable since it prevents jumps of the process $X_t$ jumping from some 
finite position $X_{t-} $  to an after jump position $X_t = X_{t- }  + \Delta ( X_{t-}) = + \infty .$

\begin{proposition}\label{prop:scale}
(1) Suppose $\Gamma (\infty )=\infty .$ Then the function $s$ introduced in \eqref{eq:s2} above is a strictly increasing version of the scale function of the process obeying $s (1)=0.$

(1.1) If additionally Assumptions \ref{As1} and \ref{ass:C} hold and if $ k(0) < \infty, $ then $ s(0 ) = - \infty $ and $ s( \infty ) = \infty , $ such that $s  $ is a space transform $ [0, \infty ) \to [ - \infty , \infty ) .$

(1.2) If Assumption \ref{ass:C} does not hold,  then 
\begin{equation*}
s_{1}(x)=\int_{x}^{\infty }\gamma (y)e^{-\Gamma (y)}/k(y)dy = s ( \infty ) - s(x) 
\end{equation*}
is a version of the scale function which is strictly decreasing, positive,
such that $s_{1}(\infty )=0.$

(2) Finally, suppose that $\Gamma (\infty )<\infty .$ Then the only scale
functions belonging to $C^1 $  are the constant ones.
\end{proposition}

\begin{remark}
\begin{enumerate}
\item[1.]
We shall see later that -- as in the case of one-dimensional diffusions, see e.g. Example 2 in Section 3.8 of Has'minskii (1980) -- the fact that  $s$ is a space transform as in item (1.1) above is related to the Harris recurrence of the process. 
\item[2.]
The assumption $\Gamma ( \infty ) < \infty $ of item (3) above corresponds to Assumption 2 of \cite{GHL} where this was the only case that we considered. As a consequence, for the GC-model considered there we did not dispose of non-constant scale functions.
\end{enumerate}
\end{remark}

\begin{proof}
To find a $C^1-$scale function $s,$ it necessarily solves 
$$
\left( \mathcal{G}s\right) \left( x\right) = -\alpha \left( x\right) s^{\prime }\left( x\right) +\beta \left( x\right)/ k\left( x\right)
\int_{x}^{\infty }k\left( y\right) s^{\prime }\left( y\right) dy=0
$$
such that for all $ x > 0, $ 
\begin{equation}
k\left( x\right) s^{\prime }\left( x\right) -\gamma \left( x\right)
\int_{x}^{\infty }k\left( y\right) s^{\prime }\left( y\right) dy=0.
\end{equation}

Putting $u^{\prime }\left( x\right) =k\left( x\right) s^{\prime }\left(
x\right) $, the above implies in particular that $u^{\prime}$ is
integrable in a neighborhood of $+\infty $ such that $u ( \infty ) $ must be
a finite number. We get $u^{\prime }\left( x\right) =\gamma \left( x\right)
\left( u\left( \infty \right) -u\left( x\right) \right). $

\textbf{Case 1 :} $u(\infty )=0$ so that $u(x) = - c_1 e^{ - \Gamma ( x) } $
for some constant $c_1, $ whence $\Gamma \left( \infty \right) =\infty .$
We obtain 
\begin{equation}
s^{\prime }\left( x\right) = c_{1}\frac{\gamma \left( x\right) }{k\left(
x\right) }e^{-\Gamma \left( x\right) }
\end{equation}
and thus

\begin{equation}  \label{eq:scale}
s\left( x\right) =c_2 + c_{1}\int_{1}^{x }\frac{\gamma \left( y\right) }{%
k\left( y\right) }e^{-\Gamma \left( y\right) }dy
\end{equation}
for some constants $c_{1}, c_2.$ Taking $ c_2 = 0 $ and $ c_1 = 1 $ gives the formula \eqref{eq:s2}, and both items (1.1) and (1.2) follow from this. 

\textbf{Case 2 :} $u(\infty )\neq 0$ is a finite number. Putting $%
v(x)=e^{\Gamma (x)}u(x),$ $v$ then solves 
\begin{equation*}
v^{\prime }(x)=u(\infty )\gamma (x)e^{\Gamma (x)}
\end{equation*}
such that 
\begin{equation*}
v(x)=d_{1}+u(\infty )e^{\Gamma (x)}
\end{equation*}
and thus 
\begin{equation*}
u(x)=e^{-\Gamma (x)}d_{1}+u(\infty ).
\end{equation*}
Letting $x\to \infty ,$ we see that the above is perfectly well-defined for
any value of the constant $d_{1},$ if we suppose $\Gamma (\infty )=\infty .$

As a consequence, $u^{\prime }(x)=-d_{1}\gamma (x)e^{- \Gamma (x)},$ leading
us again to the explicit formula 
\begin{equation}\label{eq:scalebis}
s(x)=c_{2} + c_{1}\int_{1}^x\frac{\gamma \left( y\right) }{k\left( y\right) }%
e^{-\Gamma \left( y\right) }dy,
\end{equation}
with $c_1 = -d_1,$ implying items (1.1) and (1.2). 

Finally, if $\Gamma ( \infty ) < \infty, $ we see that we have to take $d_1
= 0 $ implying that the only scale functions in this case are the constant
ones.
\end{proof}

\begin{example}
In the linear case example with $\beta \left( x\right) =\beta _{1}>0$, $%
\alpha \left( x\right) =\alpha _{1}x$, $\alpha _{1}>0$ and $k\left( y\right)
=e^{-y},$ with $\gamma _{1}=\beta _{1}/\alpha _{1},$  Assumption \ref{ass:C} is satisfied such that 
\begin{equation*}
s \left( x\right) =\gamma _{1}\int_{1}^{x}y^{-\left( \gamma _{1}+1\right)
}e^{y}dy
\end{equation*}
which is diverging both as $x\rightarrow 0$ and as $x\to +\infty .$ Notice
that $0$ is inaccessible for this process, i.e., starting from a strictly
positive position $x>0,$ $X_{t}$ will never hit $0.$ Defining the process on
the state space $(0,\infty ),$ invariant under the dynamics, the process is
recurrent and even positive recurrent as we know from the Gamma shape of its
invariant speed density.
\end{example}

\subsection{Hitting times}

Fix $a < x < b.$ In what follows we shall be interested in hitting times
of the positions $a$ and $b,$ starting from $x,$ under the condition $\Gamma
( \infty ) = \infty .$ Due to the asymmetric structure of the process
(continuous motion downwards and up-moves by jumps only), these times are
given by 
\begin{equation*}
\tau_{x, b } = \inf \{ t > 0 : X_t = b \} = \inf \{ t > 0 : X_t \le b ,
X_{t- } > b \}
\end{equation*}
and 
\begin{equation*}
\tau_{x, a } = \inf \{ t > 0 : X_t = a \} = \inf \{ t > 0 \  :  \ X_t \le a \} .
\end{equation*}
Obviously, $\tau_{a,a} = \tau_{b, b } = 0 .$

Let $T=\tau _{x,a}\wedge \tau _{x,b}.$ Contrarily to the study of processes with continuous trajectories, it is not clear that $T
<\infty $ almost surely. Indeed, starting from $x,$ the process could jump
across the barrier of height $b$ before hitting $a$ and then never enter
the interval $[0,b]$ again. So we suppose in the sequel that $T<\infty $
almost surely. Then 
\begin{equation*}
\mathbf{P}\left( \tau _{x,a}<\tau _{x,b}\right) +\mathbf{P}\left( \tau
_{x,b}<\tau _{x,a}\right) =1.
\end{equation*}

\begin{proposition}\label{prop:exit}
Suppose $\Gamma (\infty )=\infty $ and that Assumption \ref{ass:C} does not hold. Let $0<a<x<b<\infty $ and
suppose that $T= \tau _{x,a}\wedge \tau _{x,b}<\infty $ almost surely. Then 
\begin{equation}
\mathbf{P}\left( \tau _{x,a}<\tau _{x,b}\right) = \frac{\int_{x}^{b}\frac{%
\gamma \left( y\right) }{k\left( y\right) }e^{{-}\Gamma \left( y\right) }dy}{%
\int_{a}^{b}\frac{\gamma \left( y\right) }{k\left( y\right) }e^{{-}\Gamma
\left( y\right) }dy}.  \label{scaleproba}
\end{equation}
\end{proposition}

\begin{proof}
Under the above assumptions,  $s_1 (X_{t})$ is a local martingale
and the stopped martingale $M_{t}=s_1(X_{T\wedge t})$ is bounded, which
follows from the fact that $X_{T\wedge t}\geq a$ together with the
observation that $s_1$ is decreasing implying that $M_{t}\leq s_1(a).$
Therefore from the stopping theorem we have 
\begin{equation*}
s_1 (x)=\mathbb{E}(M_0)=\mathbb{E}(s_1(X_0))=\mathbb{E}(M_T).
\end{equation*}
Moreover, 
\begin{equation*}
\mathbb{E}(M_T)=s_1(a)\mathbf{P}(T=\tau _{x,a})+s_1(b)\mathbf{P}(T=\tau
_{x,b}).
\end{equation*}
But $\{T=\tau _{x,a}\}=\{\tau _{x,a}<\tau _{x,b}\}$ and $\{T=\tau
_{x,b}\}=\{\tau _{x,b}<\tau _{x,a}\},$ such that 
\begin{equation*}
s_1(x)=s_1(a)\mathbf{P}(\tau _{x,a}<\tau _{x,b})+s_1 (b)\mathbf{P}(\tau
_{x,b}<\tau _{x,a}).
\end{equation*}
\end{proof}

\begin{remark}
- See \cite{KS} for similar arguments in a particular case of a constant flow and exponential jumps.

- We stress that it is not possible to deduce the above formula without imposing the 
existence of $ s_1$ (that is, if Assumption \ref{ass:C} holds such that $s_1$ is not well-defined). Indeed, if we would want to consider the local martingale $ s ( X_t) $ instead, the stopped martingale $ s ( X_{t \wedge T}) $ is not bounded since $ X_{t \wedge T} $
might take arbitrary values in $ ( a, \infty ) , $ 
such that it is not possible to apply the stopping rule. 
\end{remark}

\begin{remark}
Let $\tau _{x,[b,\infty )}=\inf \{t>0:X_{t}\geq b\}$ and $\tau
_{x,[0,a]}=\inf \{t>0:X_{t}\le a\}$ be the entrance times to the intervals $%
[b,\infty )$ and $[0,a].$ Observe that by the structure of the process,
namely the continuity of the downward motion, 
\begin{equation*}
\{\tau _{x,a}<\tau _{x,b}\}\subset \{\tau _{x,a}<\tau _{x,[b,\infty
)}\}\subset \{\tau _{x,a}<\tau _{x,b}\}.
\end{equation*}
Indeed, the second inclusion is trivial since $\tau _{x,[b,\infty )}\le \tau
_{x,b}.$ The first inclusion follows from the fact that it is not possible
to jump across $b$ and then hit $a$ without touching $b.$ Therefore, %
\eqref{scaleproba} can be rewritten as 
\begin{equation}
\mathbf{P}\left( \tau _{x,[0,a]}<\tau _{x,[b,\infty )}\right) =\frac{%
s_{1}\left( x\right) -s_{1}\left( b\right) }{s_{1}\left( a\right)
-s_{1}\left( b\right) } .  \label{scalebis}
\end{equation}

Now suppose that the process does not explode in finite time, that is, during each finite time interval, almost surely, only a finite number of jumps appear. In this case  $\tau
_{x,[b,\infty )}\overset{}{\to }+\infty $ as $b\to \infty .$ Then, letting $ b \to \infty $ in \eqref{scalebis}, 
we obtain for any $%
a>0,$ 
\begin{equation}
\mathbf{P}\left( \tau _{x,a}<\infty \right) ={\frac{s_{1}(x)}{s_{1}(a)}}=%
\frac{\int_{x}^{\infty }\frac{\gamma \left( y\right) }{k\left( y\right) }e^{{%
-}\Gamma \left( y\right) }dy}{\int_{a}^{\infty }\frac{\gamma \left( y\right) 
}{k\left( y\right) }e^{{-}\Gamma \left( y\right) }dy}<1,  \label{eq:scale2}
\end{equation}
since $ \int_a^x \frac{\gamma \left( y\right) 
}{k\left( y\right) }e^{{-}\Gamma \left( y\right) }dy > 0 $ by assumptions on $ k $ and $ \gamma .$ 
\end{remark}

As a consequence, we obtain the following

\begin{proposition}
\label{cor:transience} Suppose that $\Gamma ( \infty ) = \infty $ and that Assumption \ref{ass:C} does not hold. Then either the processes explodes 
in finite time with positive probability or
it  is transient at $+\infty , $ i.e. for
all $a<x,$ $\tau _{x,a}=\infty $ with positive probability.
\end{proposition}

\begin{proof}
Let $ a < x < b $ and $T = \tau_{x, a } \wedge \tau_{x, b } $ and suppose that almost surely, the process does not explode in finite time. We show  that in this case, $ \tau_{x, a } = \infty $ with 
positive probability. 

Indeed, suppose that $ \tau_{x, a } < \infty $ almost surely. Then \eqref{scaleproba} holds, and letting $b \to \infty, $ we obtain \eqref{eq:scale2} implying that $ \tau_{x, a } = \infty $ with 
positive probability which is a contradiction.
\end{proof}

\begin{example}Choose $\alpha \left( x\right) =x^{2}, $  $\beta \left(
x\right) =1+x^{2}$ so that $\gamma \left( x\right) =1+1/x^{2}$%
 and $\Gamma \left( x\right) =x-1/x $  with $\Gamma \left(
0\right) =-\infty $ and $\Gamma \left( \infty \right) =\infty .$

Choose also $k\left( x\right) =e^{-x/2}. $ Then Assumption \ref{ass:C} is violated: the survival function of the big upward jumps decays too 
slowly in comparison to the decay of $ e^{ - \Gamma}.$ The speed density
is  
\begin{equation*}
\pi \left( x\right) =C\frac{k\left( x\right) e^{\Gamma \left( x\right) }}{%
\alpha \left( x\right) }=Cx^{-2}e^{x/2 -1/x}
\end{equation*}
which is integrable at $ 0 $  but not at $ \infty .$ The process explodes (has an infinite number of upward jumps in finite time) with positive probability. 
\end{example}

\subsection{First moments of hitting times}

Let $a > 0 .$ We are seeking for positive solutions of 
\begin{equation*}
\left( \mathcal{G}\phi _{a}\right) \left( x\right) =-1, x \geq a,
\end{equation*}
with boundary condition $\phi _{a}\left( a\right) =0.$ The above is equivalent to 
\begin{equation*}
\left( \mathcal{G}\phi _{a}\right) \left( x\right) =-\alpha \left( x\right)
\phi _{a}^{\prime }\left( x\right) +\frac{\beta \left( x\right) }{k\left(
x\right) }\int_{x}^{\infty }k\left( y\right) \phi _{a}^{\prime }\left(
y\right) dy=-1.
\end{equation*}
This is also 
\begin{equation*}
-k\left( x\right) \phi _{a}^{\prime }\left( x\right) +\gamma \left( x\right)
\int_{x}^{\infty }k\left( y\right) \phi _{a}^{\prime }\left( y\right) dy=-%
\frac{k\left( x\right) }{\alpha \left( x\right) }.
\end{equation*}
Putting $U\left( x\right) :=\int_{x}^{\infty }k\left( y\right) \phi
_{a}^{\prime }\left( y\right) dy$, the latter integro-differential equation
reads $U^{\prime }\left( x\right) =-\gamma \left( x\right) U\left( x\right) -%
\frac{k\left( x\right) }{\alpha \left( x\right) }. $ Supposing
that $\int^{+\infty} \pi (y ) dy < \infty $ (recall \eqref{C5}), this leads
to 
\begin{eqnarray*}
U\left( x\right) &=&e^{-\Gamma \left( x\right) }\int_{x}^{\infty }e^{\Gamma
\left( y\right) }\frac{k\left( y\right) }{\alpha \left( y\right) }dy , \\
-U^{\prime }\left( x\right) &=&\gamma \left( x\right) e^{-\Gamma \left(
x\right) }\int_{x}^{\infty }e^{\Gamma \left( y\right) }\frac{k\left(
y\right) }{\alpha \left( y\right) }dy+\frac{k\left( x\right) }{\alpha \left(
x\right) }=k\left( x\right) \phi _{a}^{\prime }\left( x\right) ,
\end{eqnarray*}
such that 
\begin{eqnarray}  \label{eq:phia}
\phi _{a}\left( x\right) &=&\int_{a}^{x}dy\frac{\gamma \left( y\right) }{%
k\left( y\right) }e^{-\Gamma \left( y\right) }\int_{y}^{\infty }e^{\Gamma
\left( z\right) }\frac{k\left( z\right) }{\alpha \left( z\right) }%
dz+\int_{a}^{x}\frac{dy}{\alpha \left( y\right) } \nonumber \\
&=&\int_{a}^{\infty }dz\pi \left( z\right) \left[ s_1\left( a\right) -s_1
\left( x\wedge z\right) \right] +\int_{a}^{x}\frac{dy}{\alpha \left(
y\right) }.
\end{eqnarray}

Notice that $[ a, \infty ) \ni x \mapsto \phi_a ( x) $ is non-decreasing and
that $\phi_a ( x) < \infty $ for all $x > a > 0$ under our assumptions.
Dynkin's formula implies that for all $x > a $ and all $t \geq 0, $ 
\begin{equation}  \label{eq:later}
\mathbf{E}_x ( t \wedge \tau_{x, a } )= \phi_a ( x)- \mathbf{E}_x ( \phi_a (
X_{t \wedge \tau_{x, a } } ) ).
\end{equation}
In particular, since $\phi_a ( \cdot ) \geq 0, $ 
\begin{equation*}
\mathbf{E}_x ( t \wedge \tau_{x, a } ) \le \phi_a ( x) < \infty ,
\end{equation*}
such that we may let $t \to \infty $ in the above inequality to obtain by
monotone convergence that 
\begin{equation*}
\mathbf{E}_x ( \tau_{x, a } ) \le \phi_a ( x) < \infty .
\end{equation*}
In a second step, we obtain from \eqref{eq:later}, using Fatou's lemma, that 
\begin{multline*}
\mathbf{E}_x ( \tau_{x, a } ) = \lim_{ t \to \infty } \mathbf{E}_x ( t
\wedge \tau_{x, a } ) = \phi_a ( x) - \lim_{ t \to \infty } \mathbf{E}_x (
\phi_a ( X_{t \wedge \tau_{x, a } } ) ) \\
\geq \phi_a ( x) -\mathbf{E}_x ( \liminf_{ t \to \infty } \phi_a ( X_{t
\wedge \tau_{x, a } } ) ) = \phi_ a (x) ,
\end{multline*}
where we have used that $\liminf_{ t \to \infty } \phi_a ( X_{t \wedge
\tau_{x, a } } ) = \phi_a ( a) = 0 .$

As a consequence we have just shown the following

\begin{proposition}\label{prop:hitting}
Suppose that $\int^{+\infty }\pi (y)dy<\infty .$ Then $\mathbf{E}_{x}(\tau
_{x,a})=\phi _{a}(x)<\infty $ for all $0<a<x,$ where $\phi _{a}$ is given as
in \eqref{eq:phia}.
\end{proposition}

\begin{remark}
The last term $\int_{a}^{x}\frac{dy}{\alpha \left( y\right) }$ in the RHS of
the expression of $\phi _{a}\left( x\right) $ in \eqref{eq:phia} is the time
needed for the deterministic flow to first hit $a$ starting from $x>a,$
which is a lower bound of $\phi _{a}\left( x\right) $. Considering the tail
function of the speed density $\pi \left( y\right) $, namely $\overline{\pi }%
\left( y\right) :=$ $\int_{y}^{\infty }e^{\Gamma \left( z\right) }\frac{%
k\left( z\right) }{\alpha \left( z\right) }dz$, the first term in the RHS
expression of $\phi _{a}\left( x\right) $ is 
\begin{equation*}
\int_{a}^{x}-ds_{1}\left( y\right) \overline{\pi }\left( y\right) =-\left[
s_{1}\left( y\right) \overline{\pi }\left( y\right) \right]
_{a}^{x}-\int_{a}^{x}s_{1}\left( y\right) \pi \left( y\right) dy,
\end{equation*}
emphasizing the importance of the couple $\left( s_{1}\left( \cdot \right)
,\pi \left( \cdot \right) \right) $ in the evaluation of $\phi _{a}\left(
x\right) $. If $a$ is a small critical value below which the population can
be considered in danger, this is the mean value of a `quasi-extinction'
event when the initial size of the population was $x$.
\end{remark}

\begin{remark}
Notice that the above discussion is only possible for couples $0<a<x,
$ since starting from $x,$ $X_{t\wedge \tau _{x,a}}\geq a$ for all $t.$ A
similar argument does not hold true for $x<b$ and the study of $\tau _{x,b}.$
\end{remark}

\subsection{Mean first hitting time of $0$}
Suppose that $ \Gamma ( 0 ) > - \infty .$ Then for flows $x_{t}\left( x\right) $ that go extinct in finite time $%
t_0\left( x\right) $, under the condition that $\int^{+\infty }\pi (y)dy<\infty ,$ one can let $a\rightarrow 0$ in the expression of $%
\phi _{a}\left( x\right) $ to obtain 
\begin{equation*}
\phi _{0}\left( x\right) =\int_{0}^{x}dy\frac{\gamma \left( y\right) }{%
k\left( y\right) }e^{-\Gamma \left( y\right) }\int_{y}^{\infty }e^{\Gamma
\left( z\right) }\frac{k\left( z\right) }{\alpha \left( z\right) }%
dz+\int_{0}^{x}\frac{dy}{\alpha \left( y\right) },
\end{equation*}
which is the expected time to eventual extinction of $X$  starting
from $x, $  that is, $\phi _{0}\left( x\right) =\mathbf{E}\tau _{x,0} .$ 

The last term $\int_{0}^{x}\frac{dy}{\alpha \left( y\right) }=t_0\left(
x\right) <\infty $ in the RHS expression of $\phi _{0}\left( x\right) $ is
the time needed for the deterministic flow to first hit $0$ starting from $%
x>0,$ which is a lower bound of $\phi _{0}\left( x\right) $.

Notice that under the condition $ t_0 ( x) < \infty $ and $ k(0 ) < \infty, $   $\int_{0  }\pi (y)dy<\infty ,$ such that $ \pi$ can be tuned into a probability. It is easy to see that $ \Gamma ( 0 ) > - \infty $ implies then that
$ \phi_0 ( x) < \infty .$

\begin{example}(Linear release at constant jump rate):
Suppose $\alpha \left( x\right) =\alpha _{1}>0${\em , }$\beta \left(
x\right) =\beta _{1}>0${\em , }$\gamma \left( x\right) =\gamma _{1}=\beta
_{1}/\alpha _{1}${\em , }$\Gamma \left( x\right) =\gamma _{1}x${\em \ }with%
{\em \ }$\Gamma \left( 0\right) =0>-\infty $. Choose $k\left( x\right)
=e^{-x}.$

State $0$ is reached in finite time $t_{0}\left( x\right) =x/\alpha _{1},$ and it
turns out to be reflecting. We have $\int_{0}\pi \left( x\right)
dx<\infty $ and $\int^{\infty }\pi \left( x\right) dx<\infty $ if and only
if $\gamma _{1}<1.$ In such a case, the first integral term in the above
expression of $\phi _{0}\left( x\right) $ is $\gamma _{1}/\left[ \alpha
_{1}\left( 1-\gamma _{1}\right) \right] x$, so that $\phi _{0}\left(
x\right) =x/\left[ \alpha _{1}\left( 1-\gamma _{1}\right) \right] <\infty .$
\end{example}

\section{Non-explosion and Recurrence}
In this section we come back to the scale function $s  $ introduced in \eqref{eq:s2} above. Despite the fact that we cannot use $s $ to obtain explicit expressions for exit probabilities, we show how we might use it to obtain Foster-Lyapunov criteria in spirit of Meyn and Tweedie (1993) that imply the non-explosion of  the process together with its recurrence under additional irreducibility properties.
 
Let $ S_1 < S_2 < \ldots < S_n < \ldots $ be the successive jump times of the process and $ S_\infty = \lim_{n \to \infty } S_n.$ We start discussing how we can use the scale function $s$ to obtain a general criterion for non-explosion of the process, that is, $ S_\infty = + \infty $ almost surely.  

\begin{proposition}\label{prop:MT} 
Suppose $ \Gamma ( \infty ) = \infty $ and suppose that Assumption \ref{ass:C} holds. Suppose also that $ \beta $ is continuous on $ [0, \infty ) .$ 
Let  $ V $ be any $\mathcal{C}^1-$function defined on $ [0, \infty ), $ such that $ V(x) = 1+s(x)  $ on $ [1, \infty ) $ and such that $V (x) \geq 1/2 $ for all $x. $ Then
$V$ is a norm-like function in the sense of Meyn and Tweedie (1993), and we have
\begin{enumerate}
\item $\mathcal{G} V (x) =0, $ $\forall x\geq 1.$
\item $\sup_{x\in [0,1]} |\mathcal{G}V (x)| < \infty.$
\end{enumerate}
As a consequence,  $ S_{\infty} = \sup_n S_n = \infty $ almost surely,  so that $ X $ is non-explosive.
\end{proposition}

\begin{proof}
We check that $ V$ satisfies the condition (CD0) of \cite{MT}. It is evident that $ V$ is norm-like since $ \lim_{x \to \infty }V( x) = 1 + \lim_{x \to \infty }  s ( x) = 1 + s ( \infty ) = \infty $ since Assumption \ref{ass:C} holds. 
Moreover, since $K(x,y)= \frac{k(y)}{k(x)} $, 
$$ \mathcal{G} V (x) = -\alpha(x)V' (x)+\frac{\beta(x)}{k(x)}\int_x^{\infty}k(y) V' (y)dy. $$
Since for all $ 1 \le x \le y , $  $V' (x)= s'(x) $ and $ V' (y) = s ' (y) , $  we have $ \mathcal{G} V  =\mathcal{G}{s}=0 $ on $[1,\infty[.$

For the second point, for $ x \in ]0,1[, $
\begin{eqnarray*}
\mathcal{G}V(x) &=& -\alpha(x)V'(x)+\frac{\beta(x)}{k(x)}\int_x^{\infty}k(y)V'(y)dy \\
									   &=& -\alpha(x)V'(x)+\frac{\beta(x)}{k(x)}\int_1^{\infty}k(y)V'(y)dy + \beta(x)\int_x^1 K(x,y)V'(y)dy.
\end{eqnarray*}
$\alpha(x)V'(x) $ is continuous and thus bounded on $[0,1]. $ Moreover, for all  $y\geq x, $ $K(x,y) \leq 1 $  implying that   $\beta(x)\int_x^1 K(x,y)V'(y)dy \leq \beta(x)\int_x^1 V'(y)dy < \infty.$ We also have for all $y \geq 1, $  $$\int_1^{\infty}k(y)V'(y)dy=\int_1^{\infty}k(y)s'(y)dy .$$ Thus, using $\mathcal{G}V(x)=\mathcal{G}s(x)=0 $ on $]0,1[ $ we have $$\frac{\beta(x)}{k(x)} \int_1^{\infty}k(y)V'(y)dy= \frac{\beta(x)}{k(x)} k(1)V'(1)/\gamma(1) < \infty $$ because $\beta $ is continuous on $[0,1]$ and $k$ taking finite values on $ (0, \infty) .$ As a consequence, 
$\sup_{x\in [0,1]} |\mathcal{G}V(x)| < \infty. $
\end{proof}

We close this subsection with a stronger Foster-Lyapunov criterion implying the existence of finite hitting time moments. 

\begin{proposition}
Suppose there exist $ x_* > 0, $ $ c > 0 $ and a positive function $V$ such that $\mathcal{G}V (x)  \leq -c $ for all $ x \geq x_* .$ Then for all $ x \geq a \geq x_*,$  
$$\mathbf{E}(\tau_{x,a}) \leq \frac{V(x)}{c}.$$ 
\end{proposition}

\begin{proof}
Using Dynkin's formula, we have for $ x \geq a \geq x_*, $ 
$$
\mathbf{E}(V(X_{t\wedge \tau_{x,a}}))= V(x)+ \mathbf{E}(\int_0^{t\wedge \tau_{x,a}} \mathcal{G}V(X_s) ds )
									   \leq V(x)-c \mathbf{E}(t\wedge \tau_{x,a} ) ,
$$
such that
$$ \mathbf{E}(t\wedge \tau_{x,a} )\leq \frac{V(x)}{c},$$ which implies the assertion, letting $t \to \infty .$ \end{proof}

\begin{example}
Suppose $\alpha (x)=1+x, $ $\beta (x)=x$ and $ 
k(x)=e^{-2x}.$ We choose $V(x)=e^{x} .$ Then
\begin{equation*}
\left( \mathcal{G}V\right) \left( x\right) =-e^{x}\leq -1,\forall x\geq 0.
\end{equation*}
As a consequence $\mathbf{E}(\tau _{x,a})<\infty ,$ for all $a>0.$ 
\end{example}

\subsection{Irreducibility and Harris recurrence}

In this section we  impose that  $ \Gamma ( \infty ) = \infty , $ such that non-trivial scale functions do exist. 
We also assume that Assumption \ref{ass:C} holds since otherwise the process is either transient at $ \infty $ or explodes in finite time. 
Then the function $ V $ introduced in Proposition \ref{prop:MT} is a Lyapunov function. 
This is {\it almost} the Harris recurrence of the process, all we need to show is some irreducibility property that we are going to check now. 

\begin{theorem}\label{theo:petite}
Suppose we are in the separable case, that $k \in C^1 $ and that $0$ is inaccessible, that is, $t_0 ( x) = \infty $ for all $x.$ Then every compact set $ C \subset ] 0, \infty [ $ is `petite` in the sense of Meyn and Tweedie. More precisely, there exist $ t > 0, $ $ \alpha \in (0, 1 ) $ and 
a probability measure $ \nu $ on $ ( \R_+ , {\mathcal{B} }( \R_+) ) ,$ such that 
$$ P_t (x, dy ) \geq \alpha \mathbf{1}_C (x) \nu (dy ) .$$ 
\end{theorem}

\begin{proof}
Suppose w.l.o.g. that $ C = [a, b ] $ with $ 0 < a < b .$ Fix any $ t > 0 .$ The idea of our construction is to impose that all processes $ X_s ( x), a \le x \le b , $ have one single common jump during $ [0, t ].$ 
Indeed, notice that for each $ x \in  C, $ the jump rate of $ X_s ( x) $ is given by $ \beta ( x_s(x) ) $ taking values in a compact set $ [ \beta_* , \beta^* ] $ where $ \beta_* = \min \{ \beta ( x_s ( x ) ), 0 \le s \le t, x \in C \} $
and $ \beta^* = \max \{ \beta ( x_s ( x ) ), 0 \le s \le t, x \in C \} .$ Notice that $ 0 < \beta_* < \beta^* < \infty $ since $ \beta $ is supposed to be positive on $ (0, \infty ).$ We then construct all processes $ X_s (x) , s \le t, x \in C , $ using the same underlying Poisson random measure 
$ M .$ It thus suffices to impose that $E$ holds, where 
$$ E = \{ M ( [0, t - \varepsilon] \times  [ 0 , \beta^* ] ) = 0 , M ( [t- \varepsilon, t] \times [0, \beta_* ] ) = 1, M ( [t- \varepsilon, t] \times ] \beta_*, \beta^* ]   ) = 0\} .$$ 
Indeed, the above implies that up to time $ t - \varepsilon, $ none of the processes $ X_s ( x) , x \in C, $ jumps. The second and third assumption imply moreover that the unique jump time, call it $S,$ of 
$ M $ within $ [ t- \varepsilon, t ] \times [ 0 , \beta^* ] $ is a common jump of all processes. 
For each value of $x \in C,$ the associated process $ X (x) $ then  chooses  a new after-jump position $y$ according to 
\begin{equation}\label{eq:transition}
 \frac{1}{k( x_S (x)) }| k'(y)| dy \mathbf{1}_{ \{ y \geq x_S ( x) \} }.
\end{equation} 

{\bf Case 1.} Suppose $k$ is strictly decreasing, that is, $|k'| (y ) > 0 $ for all $y.$ Fix then any open ball $ B \subset [ b, \infty [ $ and notice that 
$\mathbf{1}_{ \{ y \geq x_S ( x) \} } \geq \mathbf{1}_B ( y ) ,$ since $ b \geq x_S (x) .$  Moreover, since $ k$ is decreasing, $ 1/ k (x_S (x) ) \geq 1/ k ( x_t (a) ).$ 
Therefore, the transition density given in \eqref{eq:transition} can be lower-bounded, independently of $x, $ by 
$$ \frac{1}{k( x_t (a)) } \mathbf{1}_B ( y ) | k'(y)| dy  = p \tilde  \nu ( dy ) , $$
where $\tilde  \nu ( dy ) = c  \mathbf{1}_B ( y ) | k'(y)| dy , $ normalized to be a probability density, and where $ p = \frac{1}{k( x_t (a) } / c .$ 
In other words, on the event $E , $ with probability $p, $ all particles choose a new and common position $ y \sim \tilde \nu ( dy ) $ and couple. 

{\bf Case 2.} $| k'| $ is different from $0$ on a ball  $ B $ (but not necessarily on the whole state space). We suppose w.l.o.g.  that $B$ has compact closure. Then it suffices to take $t$ sufficiently large in the first step such that  
$ x_{t- \varepsilon } (b) < \inf B .$ Indeed, this implies once more that $ \mathbf{1}_B ( y ) \le  \mathbf{1}_{ \{ y \geq x_s ( x) \}} $ for all $ x \in C $ and for $s$ the unique common jump time.

{\bf Conclusion.} In any of the above cases, let $ \bar b := \sup \{ x : x \in B \} < \infty $ and restrict the set $ E$ to 
$$ E' = E \cap \{ M ( [ t - \varepsilon ] \times ] \beta_* , \bar b ] = 0 \} .$$
Putting $ \alpha := p \, \Pr ( E' ) $ and 
$$ \nu (dy ) = \int_{ t _ \varepsilon }^t {\cal L} ( S | E' ) (ds) \int \tilde \nu  (dz) \delta_{ x_{ t- s} (z) } (dy )$$
then allows to conclude.  
\end{proof}

\begin{remark}
If $\beta $ is continuous on $ [0, \infty ) $ with $ \beta (0 ) > 0 $ and if moreover $k(0) < \infty, $ the above construction can be extended to any compact set of the form $ [0, b ], b < \infty $ and to the case where $ t_0 ( x) < \infty .$  
\end{remark}

As a consequence of the above considerations we obtain the following theorem.
\begin{theorem}\label{theo:harris}
Suppose that $ \beta ( 0 ) > 0, k ( 0) < \infty ,$ that $ \Gamma ( \infty ) = \infty  $ and moreover that Assumption \ref{ass:C} holds. Then the process is recurrent in the sense of Harris and its unique invariant measure is given by $ \pi.$ 
\end{theorem}

\begin{proof}
Condition (CD1) of Meyn and Tweedie (1993) holds with $ V$ given as in Proposition \ref{prop:MT} and with compact set $ C = [0, 1 ].$ By Theorem \ref{theo:petite}, all compact sets are `petite'. Then Theorem 3.2 of \cite{MT} allows to conclude. 
\end{proof}

\begin{example}
A meaningful recurrent example consists of choosing $\beta \left(x\right) =\beta _{1}/x,$ $\beta _{1}>0$ (the surge rate decreases like $1/x $), $\alpha \left( x\right) =\alpha _{1}/x $ (finite time extinction of $x_{t}$), $\alpha _{1}>0,$ and $k\left( y\right) =e^{-y}.$ In this case, all compact sets are `petite`. Moreover, with $\gamma _{1}=\beta _{1}/\alpha _{1}$, $\Gamma \left( x\right) =\gamma_{1}x$ and 
\begin{equation*}
\pi\left( y\right)  =\frac{y}{\alpha _{1}}e^{\left(\gamma_{1}-1\right)y}
\end{equation*} 
which can be tuned into a probability density if
 $\gamma _{1} < 1 . $ This also implies that Assumption \ref{ass:C} is satisfied such that $ s $ can be used to define a Lyapunov function.  The associated process $X$ is positive recurrent if $ \gamma_1 < 1 ,$ null-recurrent if $ \gamma_1 = 1.$
\end{example}

\section{The embedded chain}

In this section, we illustrate some of the previously established theoretical
results by simulations of the embedded chain that we are going to define now.
Defining $ T_1 = S_1, T_n = S_n - S_{n-1} , n \geq 2, $ the successive inter-jump waiting times, we have
\begin{eqnarray*}
\mathbf{P}\left( T_{n}\in dt,X_{S_{n}}\in dy\mid X_{S_{n-1}}=x\right)
&=&dt\beta \left( x_{t}\left( x\right) \right) e^{-\int_{0}^{t}\beta \left(
x_{s}\left( x\right) \right) ds}K\left( x_{t}\left( x\right) ,dy\right) \\
&=&dt\beta \left( x_{t}\left( x\right) \right) e^{-\int_{x_{t}\left(
x\right) }^{x}\gamma \left( z\right) dz}K\left( x_{t}\left( x\right)
,dy\right) .
\end{eqnarray*}
The embedded chain is then defined through  $Z_{n}:=X_{S_{n}}, n \geq 0.$ If $0$ is not absorbing, for all $
x\geq 0$ 
\begin{eqnarray*}
\mathbf{P}\left( Z_{n}\in dy\mid Z_{n-1}=x\right) &=&\int_{0}^{\infty }dt\beta \left( x_{t}\left( x\right) \right) e^{-\int_{x_{t}\left( x\right) }^{x}\gamma \left( z\right) dz}K\left( x_{t}\left( x\right) ,dy\right) \\
&=&e^{-\Gamma \left( x\right) }\int_{0}^{x}dz\gamma \left( z\right)
e^{\Gamma \left( z\right) }K\left( z,dy\right) ,
\end{eqnarray*}
where the last line is valid for $ x> 0 $ only, and only if $ t_0 ( x) = \infty .$ This implies that  
 $Z_{n}$ is a time-homogeneous discrete-time Markov chain on 
$\left[ 0,\infty \right) $\textbf{.}

\begin{remark}
\label{rem7} $\left( S_{n},Z_{n}\right) _{n\geq 0}$ is also a discrete-time
Markov chain on $\Bbb{R}_{+}^{2}$ with transition probabilities given by 
\begin{equation*}
\mathbf{P}\left( S_{n}\in dt\mid Z_{n-1}=x,S_{n-1}=s\right) =dt\beta \left(
x_{t-s}\left( x\right) \right) e^{-\int_{0}^{t-s}\beta \left( x_{s^{\prime
}}\left( x\right) \right) ds^{\prime }}\text{, }t\geq s,
\end{equation*}
and 
\begin{equation*}
\mathbf{P}\left( Z_{n}\in dy\mid Z_{n-1}=x,S_{n-1}=s\right)
=\int_{0}^{\infty }dt\beta \left( x_{t}\left( x\right) \right)
e^{-\int_{x_{t}\left( x\right) }^{x}\gamma \left( z\right) dz}K\left(
x_{t}\left( x\right) ,dy\right) 
\end{equation*}
independent of $s.$ Note
\begin{equation*}
\mathbf{P}\left( T_{n}\in d\tau \mid Z_{n-1}=x\right) =d\tau \beta \left(
x_{\tau }\left( x\right) \right) e^{-\int_{0}^{\tau }\beta \left(
x_{s}\left( x\right) \right) ds},\tau \geq 0 .
\end{equation*}
\end{remark}
Coming back to the marginal $Z_{n}$ and
assuming $\Gamma \left( 0\right) =-\infty ,$ the arguments of Sections 2.5 and 2.6 in \cite{GHL} imply that  
\begin{equation*}
\mathbf{P}\left( Z_{n}>y\mid Z_{n-1}=x\right) =e^{-\Gamma \left( x\right)
}\int_{0}^{x}dz\gamma \left( z\right) e^{\Gamma \left( z\right)
}\int_{y}^{\infty }K\left( z,dy^{\prime }\right)
\end{equation*}
\begin{equation}
=1-e^{\Gamma \left( x\wedge y\right) -\Gamma \left( x\right) }+e^{-\Gamma
\left( x\right) }\int_{0}^{x\wedge y}dz\gamma \left( z\right) e^{\Gamma
\left( z\right) }K\left( z,y\right) .  \label{S1}
\end{equation}
To obtain the last line, we have used that  $K\left( z,y\right) =1$\ for all $y\leq z$\ and, whenever $z<y$,
we have split the second integral in the first line into the two
parts corresponding to ($z<y\leq x$ and $z\leq x<y$). %Note 
%\begin{equation*}
%\mathbf{E}\left( Z_{n}\mid Z_{n-1}=x\right) =\int_{0}^{\infty }\mathbf{P}%
%\left( Z_{n}>y\mid Z_{n-1}=x\right) dy
%\end{equation*}
%so that

%\begin{equation*}
%\mathbf{E}\left( Z_{n}\mid Z_{n-1}=x\right) -x=e^{-\Gamma \left( x\right)
%}\int_{0}^{x}dz\gamma \left( z\right) e^{\Gamma \left( z\right) }\left[
%\int_{z}^{\infty }\left( y^{\prime }-x\right) K\left( z,dy^{\prime }\right)
%-1/\gamma \left( z\right) \right] {\color{blue}????}
%\end{equation*}
%The first part of the bracket concerns the average move up, while the second
%part the average move down.

To simulate the embedded chain, we have to decide first if, given $Z_{n-1}=x$%
, the forthcoming move is down or up.

- A move up occurs with probability given by 
$ \mathbf{P}\left( Z_{n}> x\mid Z_{n-1}=x\right)=e^{-\Gamma \left( x\right)
}\int_{0}^{x}dz\gamma \left( z\right) e^{\Gamma \left( z\right) } K\left(
z,x\right) \mathbf{.}
$

- A move down occurs with complementary probability.

As soon as the type of move is fixed (down or up), to decide where the
process goes precisely, we must use the inverse of the corresponding
distribution function (\ref{S1}) (with $y\leq x$\ or $y>x$), conditioned on
the type of move.\newline

\begin{remark}
%
%$\left( i\right) $ If the jump kernel $K\left( z,y\right) $\ is increasing
%in $z$\ for each fixed $y$, then, from (\ref{S1}), the embedded chain is
%stochastically monotone in that, for each fixed $y$, $\mathbf{P}\left(
%Z_{n}\leq y\mid Z_{n-1}=x\right) $\ is decreasing in $x$. Note
%that 
%\begin{eqnarray*}
%\mathbf{P}\left( Z_{n}\in dy\mid Z_{n-1}=x\right) &=&e^{-\Gamma \left(
%x\right) }\int_{0}^{x}dz\gamma \left( z\right) e^{\Gamma \left( z\right)
%}K\left( z,dy\right) , \\
%&=&\mathbf{E}K\left( G\left( x\right) ,dy\right) .
%\end{eqnarray*}
%
%$\left( ii\right) $ 
If state $0$\ is absorbing, Eq. (\ref{S1}) is valid only
when $x>0,$\ and the boundary condition $\mathbf{P}\left( Z_{n}=0\mid
Z_{n-1}=0\right) =1$\ should be added. 

\end{remark}

In the following simulations, as before,  we work in the separable case $K(x,y)=\frac{k(y)%
}{k(x)} ,$ where we choose $k(x)=e^{-x} $ in the first simulation and $k(x)=1/(1+x^2) $ in the second. 
Moreover, we take $\alpha(x)=\alpha_1 x^a$ and $%
\beta(x)=\beta_1 x^b $ with $\alpha_1=1, $ $a=2,$ $\beta_1=1$  and $b=1.$ In these cases, there is no finite time extinction of the process $x_t(x) ;$ that is, in both cases, state $0$ is not accessible.

%The graph at the left is a simulation of a process $X_t$ such that $%
%\alpha(x)=\alpha_1 x^a$, $\beta(x)=x^b $ with $\alpha_1=1, $ $a=2$ and $b=1 $
%and $k(x)=e^{-x}. $ The graph at the right is a simulation of a process $X_t$
%where $\alpha(x)=\alpha_1 x^a$, $\beta(x)=3x^b $ with $\alpha_1=1, $ $a=1$
%and $b=0 $ and $k(x)=e^{-2x}. $ In both cases, state $0$ is not accessible.

\begin{center}
\includegraphics[scale=0.75]{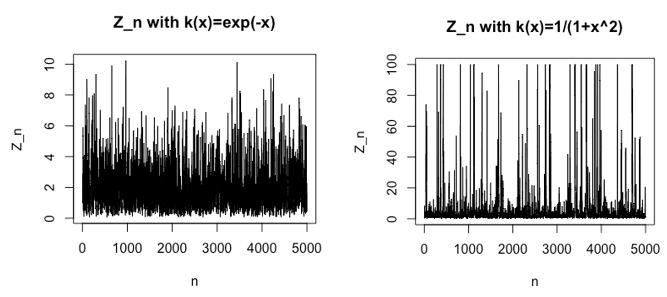}
\end{center}

Notice that in accordance with the fact that $k(x)=1/(1+x^2)$ has slower decaying tails than $k(x)= e^{-x},$ the process with jump distribution $k(x)=1/(1+x^2)$ has higher maxima than the process with  $k(x)= e^{-x}$ .

The graphs above do not provide any information about the jump
times. In what follows we take this additional information into account and simulate the values $Z_{n}$ of the embedded process as a function of the
jump times $S_{n} .$ To do so we must calculate the distribution $\mathbf{P}\left(
S_{n}\leq t\mid X_{S_{n-1}}=x,S_{n-1}=s\right) .$ Using Remark \ref{rem7} we have 
\begin{multline*}
\mathbf{P}\left( S_n\leq t \mid X_{S_{n-1}}=x, S_{n-1}=s\right) =
\int_{s}^{t}dt^{\prime}\beta \left( x_{t^{\prime}-s}\left( x\right) \right)
e^{-\int_{0}^{t^{\prime}-s}\beta \left( x_{s^{\prime }}\left( x\right)
\right) ds^{\prime }} \\
=\int_{0}^{t-s}du\beta \left( x_{u}\left( x\right) \right)
e^{-\int_{0}^{u}\beta \left( x_{s'}\left( x\right)
\right) ds'} 
= 1- e^{-\int_{0}^{t-s}\beta \left( x_{s'}\left( x\right) \right)
ds'} 
= 1- e^{-[\Gamma(x)-\Gamma(x_{t-s}(x))]}.
\end{multline*}

The simulation of the jump times $S_n $ then goes through a simple inversion
of the conditional distribution function $\mathbf{P}\left( S_n\leq t \mid
X_{S_{n-1}}=x, S_{n-1}=s\right) .$ In the following simulations we use the same parameters as in the
previous simulations.

\begin{center}
\includegraphics[scale=0.8]{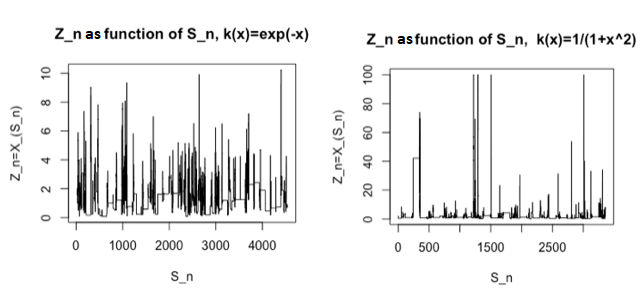}
\end{center}

The above graphs give the positions $Z_n$ as a function of the jump times $S_n$.
The waiting times between successive jumps are longer in the first process than in the second
one. Since we use the same jump rate function in both processes and since this rate is 
an increasing function of the positions, this is due to the fact that jumps lead to higher values in the second process than in the first such that jumps occur more frequently.

\begin{center}
\includegraphics[scale=0.7]{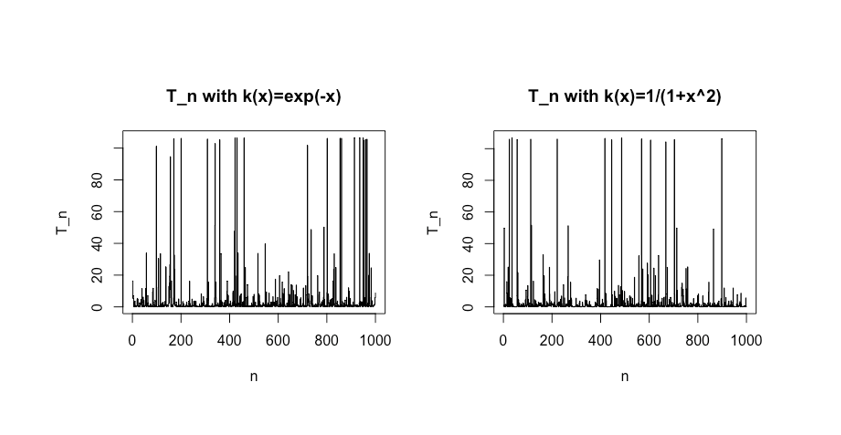}
\end{center}

These graphs represent the sequence $T_{n}=S_{n}-S_{n-1} $ of the
inter jump waiting times for the two processes showing once more that these waiting times are indeed
longer in the first process than in the second. 

%\begin{proposition}
%Suppose that $X_{t}$\ is Harris recurrent having invariant probability measure $\pi $ such that $ \pi ( \beta ) \in (0, \infty ) .$
%Let $S_{n},n\geq 1,$\ be the successive
%jump times of the process. Then $Z_{n}$\ is Harris recurrent with invariant
%measure $\pi ^{Z}$\ given by 
%\begin{equation*}
%\pi ^{Z}(g)=\frac{1}{\pi (\beta )}\pi (\beta Kg),
%\end{equation*}
%for any $g:\Bbb{R}^{N}\to \Bbb{R}$\ measurable and bounded, where 
%\begin{equation*}
%\beta Kg(x)=\beta (x)\int_{x}^{\infty }K(x,dy)g(y).
%\end{equation*}
%\end{proposition}
%After an integration by parts, it follows from the latter
%that for separable kernels $ K ( x,y ) = k(y) / k(x), $ 
%\begin{equation*}
%\text{if }k\left( x\right) e^{-\Gamma \left( x\right) }\rightarrow 0\text{
%as }x\rightarrow 0\text{ and }x\rightarrow \infty \text{, then }\pi ^{Z}(dx)=%
%\frac{e^{\Gamma \left( x\right) }dk\left( x\right) }{\int_{0}^{\infty
%}e^{\Gamma \left( x\right) }dk\left( x\right) }
%\end{equation*}
%is the explicit representation of the invariant measure of $Z.$

\section{The extremal record chain}

Of interest are the upper record times and values sequences of $Z_{n}$,
namely 
\begin{eqnarray*}
R_{n} &=&\inf \left( r\geq 1:r>R_{n-1},Z_{r}>Z_{R_{n-1}}\right) \\
Z_{n}^{*} &=&Z_{R_{n}}.
\end{eqnarray*}
Unless $X$ (and so $Z_{n}$) goes extinct, $%
Z_{n}^{*}$ is a strictly increasing sequence tending to $\infty $%
\textbf{.}

Following \cite{Adke}, with $\left( R_{0}=0,Z_{0}^{*}=x\right) $,\ $\left(
R_{n},Z_{n}^{*}\right) _{n\geq 0}$\ clearly is a Markov chain with
transition probabilities for $y>x$%
\begin{eqnarray*}
\overline{P}^{*}\left( k,x,y\right) &:=&\mathbf{P}\left(
R_{n}=r+k,Z_{n}^{*}>y\mid R_{n-1}=r,Z_{n-1}^{*}=x\right) \\
&=&\overline{P}\left( x,y\right) \text{ if }k=1 \\
&=&\int_{0}^{x}...\int_{0}^{x}\prod_{l=0}^{k-2}P\left( x_{l},dx_{l+1}\right) 
\overline{P}\left( x_{k-1},y\right) \text{ if }k\geq 2,
\end{eqnarray*}
where $P\left( x,dy\right) =\mathbf{P}\left( Z_{n}\in dy\mid
Z_{n-1}=x\right) ,$\ $\overline{P}\left( x,y\right) =\mathbf{P}\left(
Z_{n}>y\mid Z_{n-1}=x\right) $ and $x_{0}=x.$

Clearly the marginal sequence $\left( Z_{n}^{*}\right) _{n\geq 0}$\ is
Markov with transition matrix 
\begin{equation*}
\overline{P}^{*}\left( x,y\right) :=\mathbf{P}\left( Z_{n}^{*}>y\mid
Z_{n-1}^{*}=x\right) =\sum_{k\geq 1}\overline{P}^{*}\left( k,x,y\right),
\end{equation*}
but the record times marginal sequence $\left( R_{n}\right) _{n\geq 0}$\ is
non-Markov. However 
\begin{equation*}
\mathbf{P}\left( R_{n}=r+k\mid R_{n-1}=r,Z_{n-1}^{*}=x\right) =\overline{P}%
^{*}\left( k,x,x\right) ,
\end{equation*}
showing that the law of $A_{n}:=R_{n}-R_{n-1}$\ (the age of the $n$-th
record) is independent of $R_{n-1}$ (although not of $Z_{n-1}^{*}$)$:$

\begin{equation*}
\mathbf{P}\left( A_{n}=k\mid Z_{n-1}^{*}=x\right) =\overline{P}^{*}\left(
k,x,x\right) ,\text{ }k\geq 1.
\end{equation*}

Of particular interest is $\left( R_{1},Z_{1}^{*}=Z_{R_{1}}\right) $, the
first upper record time and value because $S_{R_{1}}$ is the first time $(X_t)_t $\ exceeds the threshold $x,$ and $Z_{R_{1}}$\ the
corresponding overshoot at $y>x$. Its joint distribution is simply ($y>x$) 
\begin{equation*}
P^{*}\left( k,x,dy\right) =\mathbf{P}\left( R_{1}=k,Z_{1}^{*}\in dy\mid
R_{0}=0,Z_{0}^{*}=x\right)
\end{equation*}
\begin{eqnarray*}
&=&P\left( x,dy\right) \text{ if }k=1 \\
&=&\int_{0}^{x}...\int_{0}^{x}\prod_{l=0}^{k-2}P\left( x_{l},dx_{l+1}\right)
P\left( x_{k-1},dy\right) \text{ if }k\geq 2 .
\end{eqnarray*}
If $y_{c}>x$\ is a critical threshold above which one wishes to evaluate the
joint probability of $\left( R_{1},Z_{1}^{*}=Z_{R_{1}}\right) $, then $%
P^{*}\left( k,x,dy\right) /\overline{P}^{*}\left( k,x,y_{c}\right) $\textbf{%
\ }for $y>y_{c}$\ is its right expression.

Note that $\mathbf{P}\left( R_{1}=k\mid R_{0}=0,Z_{0}^{*}=x\right) =\overline{P}%
^{*}\left( k,x,x\right) $ and also that 
$$P^{*}\left( x,dy\right) :=\mathbf{P}\left(
Z_{1}^{*}\in dy\mid Z_{0}^{*}=x\right) =\sum_{k\geq 1}P^{*}\left(
k,x,dy\right) .$$

Of interest is also the number of records in the set $\left\{
0,...,N\right\} :$%
\begin{equation*}
\mathcal{R}_{N}:=\#\left\{ n\geq 0:R_{n}\leq N\right\} =\sum_{n\geq 0}%
\mathbf{1}_{\left\{ R_{n}\leq N\right\} } .
\end{equation*}

{\color{blue} }

The following graphs represent the records of the two processes as a function of the ranks of the records, the one at
the left with $k(x) = e^{-x}$ and the one at the right with $k(x) = 1/(1+x^2)$. We can notice that there are more records in the first graph than in the second.
Records occur more frequently in the first graph than in the second. On the other hand, the heights of the records are much lower in the first graph than in the second.

\begin{center}
\includegraphics[scale=0.65]{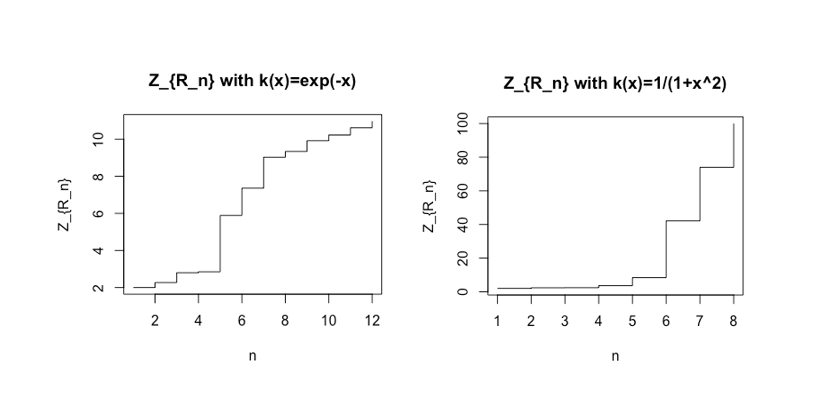}
\end{center}

The following graphs give $Z_{R_n} $ as a function of $R_n$. We remark that
the gap between two consecutive records decreases over the time whereas the
time between two consecutive records becomes longer. In other words, the higher
is a record, the longer it takes to surpass it statistically.

\begin{center}
\includegraphics[scale=0.65]{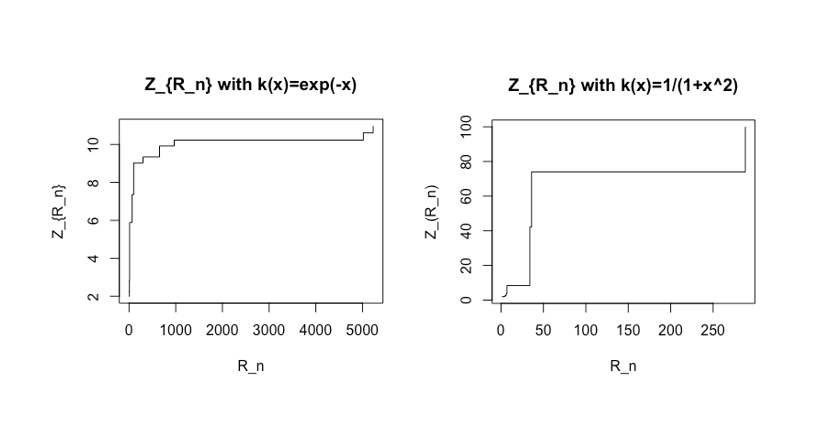}
\end{center}

The following graphs give  $A_n= R_n-R_{n-1} $ as a function of $n. $ The differences between two consecutive records are much greater in the first graph than in the second. It is also noted that the maximum time gap is reached between the penultimate record and the last record.

\begin{center}
\includegraphics[scale=0.65]{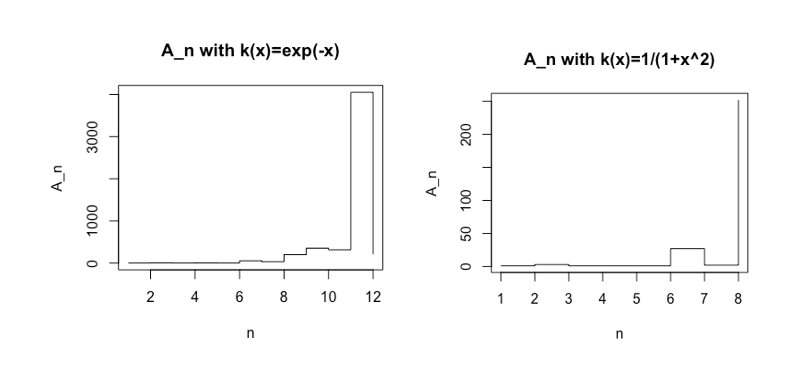}
\end{center}

The following graphs give the obtained records from the simulation of the
two processes, as a function of  time. We can remark that the curve is
slowly increasing with the time. In fact, to reach the $12$th record, the
first simulated process has needed $5500$ units of time and analogously, the
second simulated process has needed $1500$ units of time to reach its $ 8$th
record.

\begin{center}
\includegraphics[scale=0.7]{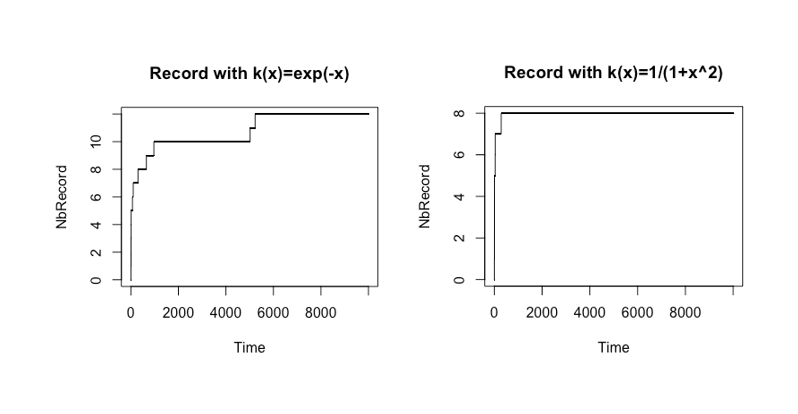}
\end{center}

\section{Decay-Surge processes and the relation with Hawkes  and Shot-noise processes}
\subsection{Hawkes processes}
In the section we study the particular case $\beta \left(
x\right) =\beta _{1}x,$ $\beta _{1}>0$ (the surge rate increases linearly
with $x$), $\alpha \left( x\right) =\alpha _{1}x$, $\alpha _{1}>0$
(exponentially declining population) and $k\left( y\right) =e^{-y}.$ In this
case, with $\gamma _{1}=\beta _{1}/\alpha _{1}$, $\Gamma \left( x\right)
=\gamma _{1}x.$ 

In this case, we remark $\Gamma (0)=0>-\infty .$ Therefore there is a
strictly positive probability that the process will never jump (in which
case it is attracted to $0$). However we have $t_{0}(x)=\infty, $  so the
process never hits $0$ in finite time. Finally, $\beta (0)=0$ implies
that state $0$ is natural (absorbing and inaccessible).

Note that for this model, 
\begin{equation*}
\pi \left( y\right) =\frac{1}{\alpha _{1}y}e^{\left( \gamma _{1}-1\right) y}
\end{equation*}
and we may take a version of the scale function given by
\begin{equation*}
s (x) =\frac{\gamma _{1}}{1-\gamma _{1}}[e^{-\left( \gamma _{1}-1\right) y}%
]_{0}^{x } = \frac{\gamma _{1}}{1-\gamma _{1}} \left( e^{\left( 1 - \gamma _{1} \right) x}- 1 \right) .
\end{equation*}
Clearly, Assumption \ref{ass:C} is satisfied if and only if $%
\gamma _{1}< 1. $ We call the case $ \gamma_1 < 1 $ {\it subcritical}, the case $ \gamma_1 > 1 $ {\it supercritical} and the case $ \gamma_1 = 1 $ {\it critical}.

{\bf Supercritical case.}  It can be shown that the process does not explode almost
surely, such that it is transient in this case (see Proposition  \ref
{cor:transience}). The speed density is neither integrable at $0$ nor at $%
\infty $.

\textbf{Critical and subcritical case.} If $\gamma_1 < 1, $ then $\pi $ is
integrable at $+\infty , $ and we find 
\begin{equation*}
\phi _{a}\left( x\right) =-\left[ s\left( y\right) \overline{%
\pi }\left( y\right) \right] _{a}^{x}-\int_{a}^{x}s \left( y\right) \pi
\left( y\right) dy+\frac{1}{\alpha _{1}}\log \frac{x}{a},
\end{equation*}
where 
\begin{equation*}
\int_{a}^{x}s \left( y\right) \pi \left( y\right) dy=\frac{\gamma _{1}}{%
\alpha _{1}\left( 1 - \gamma _{1} \right) }\left[ \ln \left( \frac{x}{a} \right) - \func{Ei}\left( \left( \gamma_{1}-1\right) x\right)  + \func{Ei}\left( \left( \gamma _{1} -1\right) a\right)
 \right] .
\end{equation*}
In the critical case $\gamma_1 = 1 $ the hitting
time of $a$ will be finite without having finite expectation.

In both critical and subcritical cases, that is, when $\gamma _{1}\leq 1, $
the process $X_t $ converges to $0$ as $t \to \infty $ as we shall show now.

Due to the additive structure of the underlying deterministic flow and the
exponential jump kernel, we have the explicit representation 
\begin{equation}  \label{eq:additive}
X_t = e^{ - \alpha_1 t } x + \sum_{ n \geq 1 : S_n \le t } e^{ - \alpha_1
(t- S_n ) } Y_n,
\end{equation}
where the $(Y_n)_{n \geq 1 } $ are i.i.d. exponentially distributed random
variables with mean $1,$ such that for all $n, $ $Y_n $ is independent of $%
S_k, k \le n , $ and of $Y_k, k < n.$ Finally, in \eqref{eq:additive}, the
process $X_t$ jumps at rate $\beta_1 X_{t- } .$

The above system is a \textit{linear Hawkes process} without immigration,
with kernel function $h( t) = e^{ - \alpha_1 t } $ and with random jump
heights $(Y_n)_{n \geq 1 } $ (see \cite{Hawkes}, see also \cite{bm}). Such a
Hawkes process can be interpreted as \textit{inhomogeneous Poisson process
with branching}. Indeed, the additive structure in \eqref{eq:additive}
suggests the following construction.

\begin{itemize}
\item  At time $0,$ we start with a Poisson process having time-dependent
rate $\beta _{1}e^{-\alpha _{1}t}x.$

\item  At each jump time $S$ of this process, a new (time inhomogeneous)
Poisson process is born and added to the existing one. This new process has
intensity $\beta _{1}e^{-\alpha _{1}(t-S)}Y,$ where $Y$ is exponentially
distributed with parameter $1,$ independent of what has happened before. We
call the jumps of this newborn Poisson process \textit{jumps of generation $1
$}.

\item  At each jump time of generation $1,$ another time inhomogeneous
Poisson process is born, of the same type, independently of anything else
that has happened before. This gives rise to jumps of generation $2.$

\item  The above procedure is iterated until it eventually stops since the
remaining Poisson processes do not jump any more.
\end{itemize}

The total number of jumps of any of the offspring Poisson processes is given
by 
\begin{equation*}
\beta_1 \mathbf{E } (Y) \int_{S}^\infty e^{ - \alpha_1 (t- S ) } dt =
\gamma_1.
\end{equation*}
So we see that whenever $\gamma_1 \le 1 , $ we are considering a subcritical
or critical Galton-Watson process which goes extinct almost surely, after a
finite number of reproduction events. This extinction event is equivalent to
the fact that the total number of jumps in the system is finite almost
surely, such that after the last jump, $X_t$ just converges to $0 $
(without, however, ever reaching it). Notice that in the  subcritical case $%
\gamma_1 < 1 ,$ the speed density is integrable at $ \infty, $ while it is
not at $0$ corresponding to absorption in $0.$

An interesting feature of this model is that it can exhibit a phase
transition when $\gamma _{1}$ crosses the value $1$.

Finally, in the case of a linear Hawkes process with immigration we have $%
\beta ( x) = \mu + \beta_1 x , $ with $\mu > 0.$ In this case, 
\begin{equation*}
\pi( y) =\frac{1}{\alpha _{1}}e^{\left( \gamma _{1}-1\right) y } y^{ \mu/
\alpha_1 - 1 }
\end{equation*}
which is always integrable in $0$ and which can be tuned into a probability
in the subcritical case $\gamma_1 < 1$ corresponding to positive recurrence.

\begin{remark}
An interpretation of the decay-surge process in terms of Hawkes processes is only possible in case of affine jump rate functions $ \beta, $ additive drift $ \alpha $ and exponential kernels $k$ as considered above. 
\end{remark}

\subsection{Shot-noise processes}
Let $h\left( t\right) $, $t\geq 0,$ with $h\left( 0\right) =1$ be a causal
non-negative non-increasing response function translating the way shocks
will attenuate as time passes by in a shot-noise process. We assume $h\left(
t\right) \rightarrow 0$ as $t\rightarrow \infty $ and 
\begin{equation}
\int_{0}^{\infty }h\left( s\right) ds<\infty .  \label{IF}
\end{equation}
With $X_{0}=x\geq 0$, consider then the linear shot-noise process 
\begin{equation}
X_{t}=x+\int_{0}^{t}\int_{{\Bbb R}_{+}}yh\left( t-s\right) \mu \left(
ds,dy\right) ,  \label{LSN}
\end{equation}
where, with $\left( S_{n};n\geq 1\right) $ the points of a homogeneous
Poisson point process with intensity $\beta ,$ $\mu \left( ds,dy\right)
=\sum_{n\geq 1}\delta _{S_{n}}\left( ds\right) \delta _{\Delta _{n}}\left(
dy\right) $ (translating independence of the shots' heights $\Delta _{n}$
and occurrence times $S_{n}$)$.$ Note that, with $dN_{s}=\sum_{n\geq
1}\Delta _{n}\delta _{S_{n}}\left( ds\right) ,$ so with $N_{t}=\sum_{n\geq
1}\Delta _{n}{\bf 1}_{\{  S_{n}\leq t \}}  $ representing a
time-homogeneous compound Poisson process with jumps' amplitudes $\Delta ,$%
\begin{equation}
X_{t}=x+\int_{0}^{t}h\left( t-s\right) dN_{s}  \label{LFSN}
\end{equation}
is a linearly filtered compound Poisson process. Under this form, it is
clear that $X_{t}$ cannot be Markov unless $h\left( t\right) =e^{-\alpha t}$%
, $\alpha >0$. We define
\begin{eqnarray*}
\nu \left( dt,dy\right)  &=&{\bf P}\left( S_{n}\in dt,\Delta _{n}\in dy\text{
for some }n\geq 1\right)  \\
&=&\beta dt\cdot {\bf P}\left( \Delta \in dy\right) .
\end{eqnarray*}
In the sequel, we shall assume without much loss of generality that $x=0$.

The linear shot-noise process $X_{t}$ has two alternative equivalent
representations, emphasizing its {\bf superposition} characteristics:

\begin{eqnarray*}
\left( 1\right) \text{ }X_{t} &=&\sum_{n\geq 1}\Delta _{n}h\left(
t-S_{n}\right) {\bf 1}_{\{ S_{n}\leq t\}} \\
\left( 2\right) \text{ }X_{t} &=&\sum_{p=1}^{P_{t}}\Delta _{p}h\left( t-%
{\cal S}_{p}\left( t\right) \right) ,
\end{eqnarray*}
where $P_t = \sum_{n\geq
1} {\bf 1}_{\{ S_{n}\leq t\}}.$

Both show that $X_{t}$ is the size at $t$ of the whole decay-surge
population, summing up all the declining contributions of the sub-families
which appeared in the past at jump times (a shot-noise or filtered Poisson
process model appearing also in Physics and Queuing theory, \cite{Sny}, \cite
{Parzen}). The contributions $\Delta _{p}h\left( t-{\cal S}_{p}\left(
t\right) \right) $, $p=1, ...P_{t},$ of the $P_{t}$ families to $X_{t}$ are
stochastically ordered in decreasing sizes.\newline

In the Markov case, $h\left( t\right) =e^{-\alpha t}$, $t\geq 0$, $\alpha >0,
$ we have
\begin{equation}
X_{t}=e^{-\alpha t}\int_{0}^{t}e^{\alpha s}dN_{s},  \label{MSN}
\end{equation}
so that 
\[
dX_{t}=-\alpha X_{t}dt+dN_{t},
\]
showing that $X_{t}$ is a time-homogeneous Markov process driven by $N_{t},$
known as the {\bf classical} linear shot-noise. This is clearly the only
choice of the response function that makes $X_{t}$ Markov. In that case, by
Campbell's formula (see \cite{Parzen}),

\begin{eqnarray*}
\Phi _{t}^{X}\left( q\right)  &:&={\bf E}e^{-qX_{t}}=e^{-\beta
\int_{0}^{t}\left( 1-\phi _{\Delta }\left( qe^{-\alpha \left( t-s\right)
}\right) \right) ds} \\
&=&e^{-\frac{\beta }{\alpha }\int_{e^{-\alpha t}}^{1}\frac{1-\phi _{\Delta
}\left( qu\right) }{u}du}\text{ where }e^{-\alpha s}=u,
\end{eqnarray*}
with 
\[
\Phi _{t}^{X}\left( q\right) \rightarrow \Phi _{\infty }^{X}\left( q\right)
=e^{-\frac{\beta }{\alpha }\int_{0}^{1}\frac{1-\phi _{\Delta }\left(
qu\right) }{u}du}\text{ as }t\rightarrow \infty .
\]

The simplest explicit case is when $\phi _{\Delta }\left( q\right) =1/\left(
1+q/\theta \right) $ (else $\Delta \sim $Exp$\left( \theta \right) $) so
that, with $\gamma =\beta /\alpha ,$ 
\[
\Phi _{\infty }^{X}\left( q\right) =\left( 1+q/\theta \right) ^{-\gamma } 
\]
the Laplace-Stieltjes-Transform of a Gamma$\left( \gamma ,\theta \right) $
distributed random variable $X_{\infty },$ with density 
\[
\frac{\theta ^{\gamma }}{\Gamma \left( \gamma \right) }x^{\gamma
-1}e^{-\theta x}\text{, }x>0. 
\]
This time-homogeneous linear shot-noise with exponential attenuation
function and exponentially distributed jumps is a decay-surge Markov process
with triple 
\[
\left( \alpha \left( x\right) =-\alpha x;\text{ }\beta \left( x\right)
=\beta ;\text{ }k\left( x\right) =e^{-\theta x}\right) . 
\]

Shot-noise processes being generically non-Markov, there is no systematic
relationship of decay-surge Markov processes with shot-noise processes. In 
\cite{EK2009}, it is pointed out that decay-surge Markov processes could be
related to the maximal process of nonlinear shot noise; see Eliazar and Klafter (2007 and 2009).

\textbf{Acknowledgments.}

T. Huillet acknowledges partial support from the ``Chaire \textit{%
Mod\'{e}lisation math\'{e}matique et biodiversit\'{e}''.} B. Goncalves and
T. Huillet acknowledge support from the labex MME-DII Center of Excellence (%
\textit{Mod\`{e}les math\'{e}matiques et \'{e}conomiques de la dynamique, de
l'incertitude et des interactions}, ANR-11-LABX-0023-01 project). This work
was funded by CY Initiative of Excellence (grant ``\textit{Investissements
d'Avenir}''ANR- 16-IDEX-0008), Project ''EcoDep'' PSI-AAP2020-0000000013.

\end{document}